\documentclass[12pt]{article}

\usepackage{amsmath}
\usepackage{graphicx}
\usepackage{enumerate}
\usepackage{natbib}
\usepackage{url} 

\newcommand{\blind}{1}

\usepackage{geometry}
\usepackage{amssymb}
\usepackage{color}
\usepackage{amsthm,bbm}
\usepackage{enumitem}
\usepackage[hyperindex=true,frenchlinks=true,colorlinks=true, urlcolor=Tan,linktocpage, pagebackref=false]{hyperref}
\usepackage{multirow}
\usepackage{threeparttable}
\usepackage{pdflscape}
\usepackage{csquotes}
\usepackage{graphicx}
\DeclareMathOperator{\E}{E}
\DeclareMathOperator{\Var}{Var}
\DeclareMathOperator{\Cov}{Cov}

\newtheorem{theorem}{Theorem}[section]
\newtheorem{definition}[theorem]{Definition}
\newtheorem{assumption}[theorem]{Assumption}

\newtheorem{proposition}[theorem]{Proposition}
\newtheorem{lemma}[theorem]{Lemma}
\newtheorem{corollary}[theorem]{Corollary}
\newtheorem{remark}[theorem]{Remark}
\newtheorem{example}[theorem]{Example}
\usepackage{framed}

\newcommand \ens[1]{\left\{ #1\right\}}
\newcommand \R{\mathbb R}

\newcommand \PP{\mathbb P}

 \newcommand{\Es}[1]{\mathbb E\left[#1\right]}

\newcommand \Z{\mathbb Z}

\newcommand \abs[1]{\left|#1\right|}

\newcommand{\pr}[1]{\left(#1\right)}

\newcommand{\1}{\mathbf{1}}
\newcommand{\e}{{\operatorname{e}}}
\newcommand{\pe}[1]{\left\lfloor #1 \right\rfloor}
\renewcommand{\leq}{\leqslant}
\renewcommand{\geq}{\geqslant}

\makeatletter
\newcommand*{\defeq}{\mathrel{\rlap{%
                     \raisebox{0.3ex}{$\m@th\cdot$}}%
                     \raisebox{-0.3ex}{$\m@th\cdot$}}%
                     =}
\makeatother

\makeatletter
\newcommand*{\eqdef}{=\mathrel{\rlap{%
                     \raisebox{0.3ex}{$\m@th\cdot$}}%
                     \raisebox{-0.3ex}{$\m@th\cdot$}}%
                     }
\makeatother

\begin{document}

\def\spacingset#1{\renewcommand{\baselinestretch}%
{#1}\small\normalsize} \spacingset{1}


\if1\blind
{
  \title{\bf   Change-point tests for the tail parameter of
 Long Memory Stochastic Volatility time series}
  \author{Annika Betken\thanks{Research supported by the German National Academic Foundation and Collaborative Research Center SFB 823 {\em Statistical modelling of nonlinear dynamic processes}}\hspace{.2cm}\\
    Faculty of Mathematics,
Ruhr-Universit\"at Bochum\\
    and \\
    Davide Giraudo\footnotemark[1]\hspace{.2cm}\\
    Faculty of Mathematics,
Ruhr-Universit\"at Bochum\\
    and \\
  Rafa{\l} Kulik\\
Department of Mathematics and Statistics,
University of Ottawa\\
}
  \maketitle
} \fi

\if0\blind
{
  \bigskip
  \bigskip
  \bigskip
  \begin{center}
    {\LARGE\bf Change-point tests for the tail parameter of
 Long Memory Stochastic Volatility time series}
\end{center}
  \medskip
} \fi

\bigskip
 \begin{abstract}
We consider a change-point test based on the Hill estimator to test for structural changes in the tail index of Long Memory Stochastic Volatility time series.
In order to determine the asymptotic distribution of the corresponding test statistic, we prove  a uniform
reduction principle for the tail empirical process
in a two-parameter Skorohod space.  It is shown that such a process displays a dichotomous behavior according to an interplay between the Hurst parameter, i.e., a parameter characterizing the dependence in the data, and the tail index.
Our theoretical results are accompanied by simulation studies
and the  analysis of  financial time series with regard to  structural changes in the tail index.
 \end{abstract}

\noindent%
\vfill

\newpage
\spacingset{1.5} 

\section{Introduction and motivation}

The tail behavior of the marginal distribution 
of time series  is of major relevance for statistics in applied  sciences such as econometrics and hydrology, where heavy-tailed data  occurs frequently.
More precisely,  time series from finance such as the log returns of exchange rates and  stock market indices display heavy tails; see \cite{mandelbrot:1963}.
Furthermore, drastic events like the financial crisis in 2008 substantiate the   importance of studying time series models that underlie financial data.
Against this background, the identification of changes in the tail behavior of data-generating stochastic processes, 
that result in an increase or decrease in the probability of extreme events, is of utmost interest. In particular, the analysis of the tail behavior of financial  data 
 may pave the way for a corresponding adjustment of risk management for capital investments
and, therefore, prevent from huge capital losses. 
Indeed, there is empirical evidence that 
the tail behavior of financial time series may change over time: \cite{quintos:2001} identify changes in the tail of Asian stock market indices,  \cite{galbraith:zernov:2004} find evidence for changes in the tail behavior of returns on U.S. equities, and \cite{werner:upper:2004} detect structural breaks in high-frequency data of Bund future returns.

\subsection{Tail index estimation and change-point problem}

  Let  $X_j$, $j\in \mathbb{N}$, be a stationary time series whose marginal tail distribution function $\bar{F}$ is regularly varying with index $-\alpha$, $\alpha>0$, i.e., 
$\PP\pr{X>x}=x^{-\alpha}L(x)$, where $L$ is slowly varying at infinity. 
Since the tail behavior of $X_j$, $j\in \mathbb{N}$, 
is primarily determined by 
 the value of the tail index $\alpha$, identifying a change in the tail of data-generating processes corresponds to testing for a change-point in this parameter.
 
In particular, this means that,   
given a set of observations $X_1, \ldots, X_n$ with
$\PP\pr{X_j>x}=x^{-\alpha_j}L(x)$, $j=1,\ldots,n$,
we aim at deciding on the testing problem $\pr{H, A}$ with
\begin{align*}
H: \ &\alpha_1=\cdots=\alpha_n
\intertext{and}
A: \
&\alpha_1=\cdots=\alpha_{k}\not=\alpha_{k+1}=\cdots=\alpha_n\\
&\text{for some } \ k\in\left\{1, \ldots, n-1\right\}.
\end{align*}

Test statistics that are designed for identifying structural changes in the tail index are naturally derived from an estimation of the tail  index $\alpha$. For some general results on tail index estimation see \cite{drees:1998a} and \cite{drees:1998b}.  
In this article, we focus on two estimators that are motivated by the fact that for a random variable $X$ with tail index $\alpha$
\begin{equation*}
\lim\limits_{u\rightarrow \infty}\Es{\log\left(\frac{X}{u}\right)\left|\right. X>u}=
\lim\limits_{u\rightarrow \infty}\frac{\Es{\log\left(\frac{X}{u}\right)\1\ens{ X>u }}}{\PP\pr{X>u}}
=\frac{1}{\alpha}\eqdef\gamma.
\end{equation*}
When we are given  a set of  observations $X_1, \ldots, X_n$, an approximation of the unknown distribution of $X$  by its empirical analogue gives the following  estimator for the tail index:
\begin{equation}\label{eq:tail_index_estimate}
\widehat{\gamma}\defeq\frac{1}{\sum_{j=1}^n\1\ens{ X_j>u_n}}\sum_{j=1}^n\log\left(\frac{X_j}{u_n}\right)\1\ens{X_j>u_n}\;,
\end{equation}
where $u_n$, $n\in \mathbb{N}$, is a sequence with $u_n\to\infty$ and $n\bar{F}\pr{u_n}\to\infty$.
Replacing the deterministic levels $u_n$ in the formula for $\widehat{\gamma}$ by $X_{n:n-k_n}$ for some $k_n$, $1\leq k_n\leq n-1$, where $X_{n:n}\geq X_{n:n-1}\geq \ldots \geq X_{n:1}$ are the  order statistics of the sample $X_1, \ldots, X_n$, yields the Hill estimator
\begin{equation*}
\widehat{\gamma}_{\rm Hill}=\frac{1}{k_n}\sum\limits_{i=1}^{k_n}\log \left(\frac{X_{n:n-i+1}}{X_{n:n-k_n}}\right).
\end{equation*}
As the most popular  estimator for the tail index, established in \cite{hill:1975},
the Hill estimator has been widely studied in the literature.
Its limiting distribution was obtained under various model assumptions,  including linear processes (\cite{resnick:starica:1997}),  $\beta$-mixing processes (\cite{drees:2000}), and Long Memory Stochastic Volatility models
(\cite{kulik:soulier:2011}).
The first article that establishes a theory for change-point tests that are based on the Hill estimator seems to be \cite{quintos:2001}.
While \cite{quintos:2001} consider independent, identically distributed observations, $\text{ARCH}$- and $\text{GARCH}$-type processes,
\cite{kim:lee:2011} and \cite{kim:lee:2012} extend their results to 
$\beta$-mixing processes and residual-based change-point tests for $\text{AR}(p)$ processes with heavy-tailed innovations.
In contrast, we study change-point tests for the tail index of Long Memory Stochastic Volatility time series 
based on the two estimators $\widehat{\gamma}$ and $\widehat{\gamma}_{\rm Hill}$. 
In fact, our results are the first to consider the change-point problem for stochastic volatility models and time series with long-range dependence.

To motivate the design of  test statistics for deciding on the change-point problem $\pr{H, A}$, we temporarily assume that
the change-point location is known, i.e., for a given $k\in \left\{1, \ldots, n-1\right\}$ we consider the testing problem $\pr{H, A_k}$ with
\begin{align*}
A_k: \
&\alpha_1=\cdots=\alpha_{k}\not=\alpha_{k+1}=\cdots=\alpha_n.
\end{align*}
For  this testing problem, 
change-point tests have been considered
in 
\cite{phillips:loretan:1990} 
and \cite{koedijk:schafgans:devries:1990}. 
In order to decide on $\pr{H, A_k}$, we compare an estimation $\widehat{\gamma}_k$ of the tail index based on the
observations $X_1, \ldots, X_k$ to an estimation  $\widehat{\gamma}_n$ of the tail index based on the whole sample $X_1, \ldots, X_n$.
This idea leads to studying the following test statistic
\begin{align*}
\Gamma_{k, n}=\frac{k}{n}\left|\frac{\widehat{\gamma}_{k}}{\widehat{\gamma}_{n}}-1\right|.
\end{align*}

Under the assumption that the change-point location is unknown under the alternative,
it seems natural to
consider the statistic $\Gamma_{k,n}$   for every potential change-point location $k$
and to decide in favor of the alternative hypothesis $A$ if the maximum of its values exceeds a predefined threshold.
As a result, a change-point test for the testing problem $\pr{H, A}$ that rests upon the estimator $\widehat{\gamma}$ defined by \eqref{eq:tail_index_estimate} bases test decisions on the values of the statistic
\begin{equation}\label{eq:defi_of_test_statistic}
\Gamma_n\defeq\sup\limits_{t\in [t_0,1]}t\left|\frac{\widehat{\gamma}_{\lfloor nt\rfloor}}{\widehat{\gamma}_n}-1\right|
\end{equation}
with  $t_0\in (0,1)$ and with the sequential version of $\widehat{\gamma}$ defined by
\begin{align}\label{eq:estimator}
\widehat{\gamma}_{\lfloor nt\rfloor}\defeq\frac{1}{\sum_{j=1}^{\lfloor nt\rfloor}
\1\{X_j>u_n\}}\sum_{j=1}^{\pe{nt}}\log\left(\frac{X_j}{u_n}\right)\1\{X_j>u_n\}.
\end{align}

Likewise, a test statistic based on the Hill estimator is given by
\begin{align*}
\widetilde \Gamma_n\defeq\sup\limits_{t\in [t_0, 1]}t\left|\frac{\widehat{\gamma}_{\rm Hill}(t)}{\widehat{\gamma}_{\rm Hill}(1)}-1\right|\;
\end{align*}
with the sequential version of $\widehat{\gamma}_{\rm Hill}$ defined by
\begin{equation*}
\widehat{\gamma}_{\rm Hill}(t)\defeq\frac{1}{\lfloor k_nt\rfloor}\sum\limits_{i=1}^{\lfloor k_nt\rfloor}\log \left(\frac{X_{\lfloor nt \rfloor:\lfloor nt\rfloor-i+1}}{X_{\lfloor nt\rfloor :\lfloor nt\rfloor-k_{\lfloor nt\rfloor}}}\right).
\end{equation*}

In this context, the most comprehensive  theory for change-point tests   is presented in \cite{hoga:2017}. The author considers a number of  test statistics based on different tail index estimators  and derives their
 asymptotic distributions under the assumption of $\beta$-mixing data generating processes.

In the following, we derive the asymptotic distribution of both estimators, i.e., $\widehat{\gamma}_{\lfloor nt\rfloor}$ and $\widehat{\gamma}_{\text{Hill}}(t)$, and the corresponding tests statistics, i.e., $\Gamma_n$ and $\widetilde \Gamma_n$, under the hypothesis of stationary time series data. For this purpose, we first prove a limit theorem for the tail empirical process of Long Memory Stochastic Volatility time series  in two parameters. This limit theorem does not necessarily relate to the change-point context. It can therefore be considered of independent interest and, thus, as the main theoretical result of our work.
Our theoretical results are accompanied by simulation studies. 
As an empirical application of our tests, we consider Standard \& Poor's 500 daily closing index  covering the period from January 2008 to December 2008, the year of the financial crisis. We identify a change in the data at exactly one day after Lehman Brothers filed for bankruptcy protection, an event which is thought to have played a major role in the unfolding of the crisis in 2007 –- 2008.

\subsection{Tail empirical process}

In order to derive the limit distribution of the tail estimators $\widehat{\gamma}_{\lfloor nt\rfloor}$ and $\widehat{\gamma}_{\rm Hill}(t)$, parametrized in $t$, and the corresponding test statistics $\Gamma_n$ and  $\widetilde \Gamma_n$, it is crucial to note that
\begin{align}
\widehat{\gamma}_{\lfloor nt\rfloor}=\frac{1}{\sum_{j=1}^{\lfloor nt\rfloor}
\1\{X_j>u_n\}}\sum_{j=1}^{\pe{nt}}\log\left(\frac{X_j}{u_n}\right)\1\{X_j>u_n\} =
\frac{1}{\widetilde{T}_n(1,t)}
\int_1^{\infty}s^{-1}\widetilde{T}_n(s,t)ds\;,
\end{align}
where
\begin{align*}
&\widetilde{T}_n(s,t)=\frac{1}{n\bar{F}(u_n)}\sum\limits_{j=1}^{\pe{nt}}\1\left\{X_j> u_ns\right\}.
\end{align*}
As a result, asymptotics of the considered statistics can be derived from a limit theorem for 
 the two-parameter tail empirical process 
\begin{equation}\label{eq:definition_empirical_processs}
e_n(s,t)\defeq\left\{\widetilde{T}_n(s,t)-T(s,t)\right\}, \ s\in [	1, \infty], \ t\in [0, 1],
\end{equation}
where $T\pr{s,t}$ does not correspond to the mean of $\widetilde{T}_n(s,t)$, but rather to the limit of that mean, i.e., to
\begin{equation}\label{eq:definition:T}
T\pr{s,t}\defeq ts^{-\alpha}.
\end{equation}

Among others,  the  tail empirical process in one parameter, i.e., $e_n(s, 1)$, $s\in [1, \infty]$,  has previously been studied in \cite{mason:1988}, \cite{einmahl:1990}, and \cite{einmahl:1992} for independent, identically distributed  observations, in \cite{rootzen:2009} for absolutely regular processes, and in
\cite{kulik:soulier:2011} for Long Memory Stochastic Volatility time series.
For the latter, the convergence of the two-parameter tail empirical process will be discussed in Section~~\ref{subsec:conv_empirical_process}.

\subsection{Long Memory Stochastic Volatility model}

A phenomenon that is often encountered in the context of financial time series corresponds to the observation that the observations seem to be uncorrelated, whereas their absolute values or higher moments  tend to be highly correlated.
Another characteristic of financial time series is the occurrence of heavy tails. In particular, the distribution of the considered data often exhibits tails that are heavier than those of a normal distribution.
The previously described features of financial data can be covered by stochastic volatility models.

\subsubsection*{Stochastic volatility model}

The Long Memory Stochastic Volatility  model that is taken as a basis of the theoretical results established in this article can be considered as a  generalization of  stochastic volatility
 models considered, for example, in \cite{taylor:1986}.
Initially, this  model had been introduced by \cite{breidt:crato:delima:1998} and, independently, by \cite{harvey:2002}.
 An overview of stochastic volatility models with long-range dependence and their basic properties  is given in \cite{deo:2006}  and in \cite{hurvich:soulier:2009}.

 Stochastic volatility time series $X_j$, $j\in\mathbb{N}$, are typically defined via
\begin{align}\label{eq: LMSV}
X_j=Z_j\varepsilon_j \ \ \text{with} \ \ Z_j=\exp\left(\frac{1}{2}Y_j\right),
\end{align}
where
$\varepsilon_j$, $j\in \mathbb{N}$, is a sequence of  independent, identically distributed random variables with mean $0$, and
$Y_j$, $j\in \mathbb{N}$, is a  Gaussian process, independent of $\varepsilon_j$, $j\in \mathbb{N}$.

While these models are often restricted to modeling a relatively fast decay of dependence in $Y_j$, $j\in \mathbb{N}$,  the so-called Long Memory Stochastic Volatility  model allows for long-range dependence.
In what follows, we will specify a corresponding dependence structure for $Y_j$, $j\in \mathbb{N}$.

\subsubsection*{Subordinated Gaussian processes}

The rate of decay  of the autocovariance function is crucial to the definition of  long-range dependence in time series.
\begin{definition}
A  (second-order) stationary, real-valued time series $Y_j$, $j\in \mathbb{Z}$, is called long-range dependent if its autocovariance function $\gamma$ satisfies
\begin{align*}
\gamma_Y(k)\defeq\Cov\pr{Y_1, Y_{k+1}}\sim k^{-D}L_{\gamma}(k), \ \ \text{as $k\rightarrow \infty$,}
\end{align*}
with $D\in \left(0, 1\right)$ for  some slowly varying function $L_{\gamma}$.  We refer to $D$ as  long-range dependence (LRD) parameter; see \cite{pipiras:taqqu:2017}, p. 17.
\end{definition}

We will focus our considerations on long-range dependent subordinated Gaussian time series.
\begin{definition}\label{def_sub_gauss}
Let $Y_j$, $j\in \mathbb{N}$, be a Gaussian process.
A process $Z_j$, $j\in \mathbb{N}$, satisfying $Z_j=G\pr{Y_j}$ for some measurable function $G\colon\mathbb{R}\longrightarrow \mathbb{R}$ is called  subordinated Gaussian process.
\end{definition}

\begin{remark}{\rm
For any particular distribution function $F$,
an appropriate choice of the transformation $G$ in Definition \ref{def_sub_gauss}
yields a subordinated Gaussian process with marginal distribution $F$.
Moreover, there exist algorithms for generating Gaussian processes that,
after suitable  transformation, yield subordinated Gaussian processes with marginal distribution $F$ and a predefined covariance structure;  see \cite{pipiras:taqqu:2017}, Section 5.8.4. To that effect, subordinated Gaussian processes constitute  a comprehensive  model for long-range dependent time series.
}
\end{remark}

The subordinated random variables $Z_j=G\pr{Y_j}$, $j\in \mathbb{N}$, can be considered as elements of the Hilbert space
$L^2\defeq L^2\pr{\mathbb{R}, \varphi(x)dx}$, i.e., the space of all measurable, real-valued  functions which are square-integrable with respect to the measure $\varphi(x)dx$ associated with the standard normal density function $\varphi$, equipped with the inner product 
\begin{align*}
\langle G_1, G_2 \rangle_{L^2}\defeq\displaystyle\int_{-\infty}^{\infty}G_1(x)G_2(x)\varphi(x)dx=\Es{G_1(Y)G_2(Y)},
\end{align*}
where  $G_1, G_2 \in L^2(\mathbb{R}, \varphi(x)dx)$ and $Y$ denotes a standard normally distributed random variable.
In order to characterize the dependence structure of subordinated Gaussian processes, we consider their expansion  in
Hermite polynomials.
 
\begin{definition}
For $n\geq 0$, the  Hermite polynomial of order $n$ is defined by
\begin{align*}
H_n(x)=(-1)^{n}\e^{\frac{1}{2}x^2}\frac{d^n}{d x^n}\e^{-\frac{1}{2}x^2}, \ x\in \mathbb{R}.
\end{align*}
\end{definition}

The sequence of Hermite polynomials constitutes an orthogonal basis of $L^2$.
In particular, it holds that
\begin{align*}
\langle H_n, H_m \rangle_{L^2}
=\begin{cases}
n! \ \ &\text{if $n=m$,}\\
0 \  \ &\text{if $n\neq m$.}
\end{cases}
\end{align*}
As a result,  every $G\in L^2(\mathbb{R}, \varphi(x)dx)$
has an expansion in Hermite polynomials, i.e., for $G\in L^2(\mathbb{R}, \varphi(x)dx)$ and $Y$ standard normally distributed, we have
\begin{align}\label{eq:Hermite_expansion}
G(Y)\overset{L^2}{=}\sum\limits_{r=0}^{\infty}\frac{J_r(G)}{r!}H_r(Y), \ \text{i.e.}, \ \lim\limits_{n\rightarrow \infty}\left\|G(Y)-\sum\limits_{r=0}^{n}\frac{J_r(G)}{r!}H_r(Y)\right \|_{L^2}
= 0,
\end{align}
where $\|\cdot\|_{L^2}$ denotes the norm induced by the inner product $\langle \cdot, \cdot\rangle_{L^2}$, and the so-called {\em Hermite coefficients} $J_r(G)$, $ r\geq 1$, are given by
\begin{align*}
J_r(G): =\langle G, H_r\rangle_{L^2}=\E G(Y)H_r(Y), \ r\geq 1.
\end{align*}

Given the Hermite expansion
\eqref{eq:Hermite_expansion}, it is possible to characterize the dependence structure of subordinated Gaussian time series  $G(Y_j)$, $j\in \mathbb{N}$.
Under the assumption that the  Gaussian sequence  $Y_j$,  $j\in \mathbb{N}$, is stationary and that $G$ is a one-to-one function,  the behavior of the autocorrelations of the transformed process is completely determined by the dependence structure of the underlying process.  However, this is not the case in general.
In fact, it holds that
\begin{align}\label{eq:cov_sub_Gaussian}
\Cov\pr{G(Y_1), G(Y_{k+1})}=\sum\limits_{r=1}^{\infty}\frac{J^{2}_r(G)}{r!}\left(\gamma_Y(k)\right)^{r},
\end{align}
where $\gamma_Y(k)$ denotes the autocovariance function of $Y_n$, $n\in \mathbb{N}$; see \cite{pipiras:taqqu:2017}.

 Under the assumption that, as $k$ tends to $\infty$,  $\gamma_Y(k)$ converges to $0$ with a certain rate, the asymptotically dominating term in the series \eqref{eq:cov_sub_Gaussian} is the summand corresponding to the smallest integer $r$ for which the Hermite coefficient $J_r(G)$ is non-zero. This index, which decisively depends on $G$, is called  Hermite rank.

\begin{definition}
Let $G\in L^2(\mathbb{R}, \varphi(x)dx)$, $\Es{G(Y)}=0$ for standard normally distributed $Y$ and let $J_r(G)$, $r\geq 1$, be the Hermite coefficients in the Hermite expansion of $G$. The smallest index $k\geq 1$ for which $J_k(G)\neq 0$ is called the {\em Hermite rank} of $G$, i.e.,
\begin{align*}
r\defeq \min\left\{k\geq 1: J_k(G)\neq 0\right\}.
\end{align*}
\end{definition}

It follows from  \eqref{eq:cov_sub_Gaussian} that subordination of long-range dependent Gaussian time series  potentially generates time series whose autocovariances  decay faster than the autocovariances of the underlying Gaussian process. In some cases, the subordinated time series is long-range dependent as well, in other cases subordination may even yield short-range dependence.
Given that
  $\Cov(Y_1, Y_{k+1})\sim k^{-D}L(k)$, as $k\rightarrow\infty$,
and given that $G\in  L^2(\mathbb{R}, \varphi(x)dx)$ is a function with Hermite rank $r$, we have
\begin{align*}
\Cov(G(Y_1), G(Y_{k+1}))\sim J_r^2(G)r!k^{-Dr}L_\gamma^r(k), \ \ \text{as $k\rightarrow \infty$.}
\end{align*}

It immediately follows  that subordinated Gaussian time series  $G(Y_j)$, $j\in\mathbb{N}$, are long-range dependent  with LRD parameter $D_G\defeq Dr$ and slowly-varying function $L_G(k)=J_r^2(G)r!L_\gamma^r(k)$ whenever $Dr<1$.

Given the previous definitions, we specify  model assumptions  that are taken as a basis for the results in the following sections.

\begin{definition}\label{model:LMSV}
Let the data generating process $X_j$, $j\in\mathbb{N}$, satisfy
\begin{align*}
X_j=Z_j\varepsilon_j, \ \ j\in \mathbb{N},
\end{align*}
where
$\varepsilon_j$, $j\in \mathbb{N}$, is a sequence of independent, identically distributed random variables with mean $0$, and
$Z_j$, $j\in \mathbb{N}$, is a long-range dependent subordinated Gaussian process  with $Z_j=\sigma(Y_j)$, $j\in \mathbb{N}$, for some stationary, long-range dependent Gaussian process $Y_j$, $j\in \mathbb{N}$, with LRD parameter $D$ and a non-negative measurable function $\sigma$ (not equal to $0$).
More precisely, assume that
$Y_j$, $j\in \mathbb{N}$, admits a linear representation with respect to an  independent, standard normally distributed sequence $\eta_k$, $k\in \mathbb{Z}$,  i.e.,
\begin{align*}
Y_j=\sum_{k=1}^{\infty}c_k\eta_{j-k},\ \ j\in \mathbb{N},
\end{align*}
with $\sum_{k=1}^\infty c_k^2=1$. Furthermore, suppose that $(\varepsilon_j,\eta_j)$, $j\in \mathbb{Z}$,
is a sequence of independent, identically distributed random 	vectors.
A sequence of random variables $X_j$, $j\in \mathbb{N}$, which satisfies the previous assumption is called a Long Memory Stochastic Volatility (LMSV) time series.
\end{definition}

\begin{remark}{\rm
The model assumptions generalize the preceding concepts of    stochastic volatility models with long-range dependence  by allowing for general subordinated Gaussian sequences $Z_j$,  $j\in \mathbb{N}$, and dependence between $Y_j$, $j\in \mathbb{N}$, and $\varepsilon_j$, $j\in \mathbb{N}$. Instead of claiming mutual independence of  $Y_j$,  $j\in \mathbb{N}$, and $\varepsilon_j$, $j\in \mathbb{N}$, the sequence of  random vectors ${\left(\eta_j, \varepsilon_j\right)}$  is assumed to be independent.
In particular, this implies that for a fixed index
$j$, the random variables  $Y_j$ and $\varepsilon_j$ are independent, while $Y_j$ may depend on  $\varepsilon_i$, $i<j$. In many cases, an LMSV model incorporating this dependence structure is   referred to as {\em LMSV with leverage}, as it allows for so-called {\em leverage effects} in financial time series. 
Not taking account of leverage,
Definition \ref{model:LMSV} corresponds to  the  LMSV model considered in \cite{kulik:soulier:2011}, while a similar model with leverage is considered in \cite{bilayi:ivanoff:kulik:2019}.
}
\end{remark}

It can be shown that  random variables $X_j$, $j\in \mathbb{N}$, satisfying Definition \ref{model:LMSV} are uncorrelated, while their squares inherit the dependence structure from the  subordinated Gaussian sequence $Z_j^2$, $j\in \mathbb{N}$.
Moreover,
 $X_j$, $j\in \mathbb{N}$, inherits the tail behavior from the  sequence $\varepsilon_j, \ j\in \mathbb{N}$, 
if the marginal distribution of the random variables $\varepsilon_j$, $j\in \mathbb{N}$, has a regularly varying right tail, i.e.,
$\overline{F}_{\varepsilon}(x)\defeq \PP\pr{\varepsilon_1>x}=x^{-\alpha}L(x)$ for some $\alpha>0$ and a slowly varying function $L$, and if
$\Es{\sigma^{\alpha+\delta}(Y_1)}<\infty$
for some $\delta >0$.
More precisely, under these assumptions 
 the following asymptotic equivalence holds:
\begin{align*}
\PP\pr{X_1>x}\sim \Es{\sigma^{\alpha}(Y_1)}\PP\pr{\varepsilon_1>x}, \ \text{as $x\rightarrow\infty$.}
\end{align*}
This result is known as  Breiman's Lemma; see \cite{breiman:1965}.
On this account, it follows that
 Definition \ref{model:LMSV} is suited for modeling the previously described characteristic features of financial time series.
In all  following sections, we will therefore assume that the data-generating process $X_j$, $j\in\mathbb{N}$, corresponds to a LMSV time series specified by Definition \ref{model:LMSV}.

\subsection{Organisation of the paper}
Equipped with the introductory remarks and definitions, we are in a position to discuss the structure of the paper. In Section \ref{sec:main} we state the technical  assumptions that are needed for our theoretical results. These are followed by the main theorem on convergence of the two-parameter tail empirical process (Theorem \ref{thm:TEP}).
Convergence of estimators of the tail index
(Corollaries \ref{cor:convergence_gamma_hat} and \ref{cor:convergence_gamma_Hill}) and the test statistics
(Corollaries \ref{cor:convergence_test_statistic} and \ref{cor:convergence_test_statistic_Hill}) are immediate consequences. Simulation studies are presented in Section \ref{sec:simulations}, while real-data analysis can be found in Section \ref{sec:data}.
All the proofs are included in Section \ref{sec:proofs}. In order to establish convergence of the two-parameter tail empirical process, we decompose it into a martingale and a long-range dependent part. The latter is dealt with in
Section \ref{sec:proof-lrd}. For the former, we establish  finite dimensional convergence
(Section \ref{sec:mtg-fidi}) using classical tools from martingale theory, while tightness of the two-parameter martingale part is handled by chaining. This is a theoretical novelty in the present context since the methods used in  related papers are not suitable (the method used in \cite{kulik:soulier:2011} cannot be applied to models with leverage, while the approach in \cite{bilayi:ivanoff:kulik:2019} is not well-suited for two-parameter processes).

\section{Main results}\label{sec:main}

\subsection{Assumptions}

In this section, we establish the assumptions guaranteeing  convergence of the two-parameter empirical process for LMSV time series.
Initially, we specify the LMSV model yielding the main assumptions for the theory: 

\begin{assumption}[Main Assumptions]\label{ass:main}
\noindent
Let $X_j=Z_j\varepsilon_j$, $j\in \mathbb{N}$,  
satisfy Definition \ref{model:LMSV} with $Z_j=\sigma(Y_j)$, $j\in \mathbb{N}$, for some stationary, long-range dependent Gaussian process $Y_j$, $j\in \mathbb{N}$, with autocovariance function $\gamma_Y(k)\defeq\Cov\pr{Y_1, Y_{k+1}}\sim k^{-D}L_{\gamma}(k),  \text{ as $k\rightarrow \infty$,}$ and 
some independent, identically distributed sequence  $\varepsilon_j$, $j\in \mathbb{N}$,
with regularly varying right tail, i.e.,   
$\overline{F}_{\varepsilon}(x)\defeq \PP\pr{\varepsilon_1>x}=x^{-\alpha}L(x)$ for some $\alpha>0$ and a slowly varying function $L$.
Moreover, let $r$ denote the Hermite rank of  $\Psi(y)\defeq \sigma^{\alpha}(y)$ and assume that
$r<1/D$. 
\end{assumption}

In the following, we  list some technical conditions  that characterize the behavior of the slowly varying function
$L$ and the moments of $\sigma\pr{Y_1}$.
For this, we introduce  another condition on the distribution function $F_{\varepsilon}$.
\begin{definition}[Second order regular variation]
Let
$\overline{F}_{\varepsilon}(x)=x^{-\alpha}L(x)$
for some $\alpha>0$ and some slowly varying function $L$ that is represented by
\begin{align*}
L(x)=c\exp\left(\int_1^x\frac{\eta(u)}{u}du\right)
\end{align*}
for some constant $c$  and a measurable function $\eta$.
Furthermore, we assume that there exists a bounded, decreasing function $\eta^{*}$ on $[0, \infty)$, regularly varying at infinity with parameter $\rho\geq 0$, i.e., $\eta^*(x)=x^{-\rho}L_{\eta^*}(x)$, such that
\begin{align*}
\abs{\eta(s)}\leq C\eta^{*}(s),
\end{align*}
for some constant $C$ and for all  $s\geq 0$. We say that $\overline{F}_{\varepsilon}$ is second order regularly varying with tail index $\alpha$ and rate function  $\eta^*$ and we write $\overline{F}_{\varepsilon} \in \text{2RV}(\alpha, \eta^*)$.
\end{definition}
Second-order regular variation allows to control the difference between $\overline{F}_{\varepsilon}$ and the function $u\mapsto u^{-\alpha}$;
see Lemma \ref{lem:control_ratio_Fz} and \ref{lem:control_ratio_difference_Fz} in the Appendix. Moreover,   the specific form of $L$ guarantees continuity of $\overline{F}_{\varepsilon}$.

\begin{assumption}[Technical Assumptions]
\label{ass:technical}
\noindent 
Suppose the main assumptions hold. Additionally, we assume that
\begin{enumerate}[label=(TA.\arabic*)]
\item\label{assump:second_order} $\overline{F}_{\varepsilon} \in \text{2RV}\pr{\alpha, \eta^*}$ \textcolor{black}{and $\eta$ is regularly varying with index $\rho$};
\item\label{assump:eta_of_un}  $\eta^{*}(u_n)=o\left(\frac{d_{n, r}}{n}+\frac{1}{\sqrt{n\overline{F}(u_n)}}\right)$, \textcolor{black}{where $d_{n, r}$ is defined by }
 \begin{align}\label{eq:dn}
d_{n, r}^2=\Var\left(\sum\limits_{j=1}^nH_r(Y_j)\right)\sim c_rn^{2-rD}L_\gamma^r(n), \ c_r=\frac{2 r!}{(1-Dr)(2-Dr)};
\end{align}
\item\label{assump:moments:isigmaY}
$\Es{\sigma^{\alpha+\max\ens{\rho, \alpha}+\varepsilon}\pr{Y_1}}<\infty$ for some $\varepsilon>0$;
\item\label{assump:moments:inverse_sigmaY_2+delta}
 $\Es{\left(\sigma\pr{Y_1}\right)^{-1}}<\infty$.
\end{enumerate}
\end{assumption}

\begin{remark}
 Assumption \ref{assump:eta_of_un} handles the bias which is created by centering the tail empirical process   not by its mean, but rather by the limit of that mean. 
\end{remark}

\begin{example}{\rm
The most commonly used second order assumption is that 
\begin{align*}
L(x)=c\exp\left(\int_1^x\frac{\eta(u)}{u}du\right)
\end{align*}
with  $\eta(s)=s^{-\alpha \beta}$ for some $\beta> 0$.
It then holds that
 $\overline{F}_{\varepsilon}(s)=C\left(s^{-\alpha}+\mathcal{O}(s^{-(\alpha(\beta+1))})\right)$,  for $s\rightarrow \infty$, for some constant $c>0$.
Furthermore, we have
\begin{align*}
\sup\limits_{s\geq s_0}\left|\frac{\overline{F}_{\varepsilon}(u_ns)}{\overline{F}_{\varepsilon}(u_n)}-s^{-\alpha}\right|=\mathcal{O}(u_n^{-\alpha\beta}).
\end{align*}
In this  case, \ref{assump:eta_of_un} can be replaced by the assumption
$u_n^{-\alpha\beta}=o\left(\frac{d_{n, r}}{n}+\frac{1}{\sqrt{n\overline{F}(u_n)}}\right)$.
}
\end{example}
\subsection{Convergence of the tail empirical process}\label{subsec:conv_empirical_process}

Recall that the tail empirical in two parameters is defined by
\begin{align*}
e_n(s,t)\defeq \frac{1}{n\bar{F}(u_n)}\sum\limits_{j=1}^{\pe{nt}}\1\left\{X_j> u_ns\right\}-ts^{-\alpha}, \ s\in [	1, \infty], \ t\in [0, 1].
\end{align*}

The following theorem establishes a characterization of its limit. 
\begin{theorem}\label{thm:TEP}
Let $X_j$, $j\in \mathbb{N}$, be a stationary time series  with marginal tail distribution function $\bar{F}$. Moreover, assume that Assumptions \ref{ass:main} and \ref{ass:technical} hold.
\begin{enumerate}
\item If $\frac{n}{d_{n, r}}=o\left(\sqrt{n\overline{F}(u_n)}\right)$,
\begin{align}\label{eq:limit-1}
\frac{n}{d_{n, r}}e_n(s, t)\Rightarrow  \frac{s^{-\alpha}}{\Es{\sigma^{\alpha}(Y_1)}}\frac{J_r(\Psi)}{r!}Z_{r, H}(t),
\end{align}
where $\Psi(y)=\sigma^{\alpha}(y)$, $r$ is the Hermite rank of $\Psi$, $Z_{r, H}$ is an $r$-th order Hermite process,  $H=1-\frac{rD}{2}$, and $d_{n, r}^{2}$ is defined in (\ref{eq:dn}).
\item If $\sqrt{n\overline{F}(u_n)}=o\left(\frac{n}{d_{n, r}}\right)$,
\begin{align}\label{eq:limit-2}
\sqrt{n\overline{F}(u_n)}e_n(s, t)\Rightarrow B({s^{-\alpha},t}),
\end{align}
 where $B$ denotes a standard Brownian sheet.
\end{enumerate}
The convergence holds in a two-parameter Skorohod space, i.e., $\Rightarrow$ denotes weak convergence in $D\left([1, \infty]\times
[0, 1]\right)$.
\end{theorem}

The dichotomy of the limiting process is explained by the decomposition of the tail empirical process into 
the sum of a   martingale and  a partial sum of long-range dependent random variables, which  can be viewed as a special case of Doob's decomposition; see Section~~\ref{subsubsec:decomposition_empirical_process}.
If $\frac{n}{d_{n, r}}=o\left(\sqrt{n\overline{F}(u_n)}\right)$,  the martingale part in the decomposition becomes negligible, such that  the limiting process arises from the convergence of the long-range dependent  part.
If $\sqrt{n\overline{F}(u_n)}=o\left(\frac{n}{d_{n, r}}\right)$, the long-range dependent part in the decomposition becomes  negligible, such that the limiting process arises from the convergence of the martingale part.
The same decomposition has already been employed in \cite{kulik:soulier:2011}, \cite{betken:kulik:2019}, and \cite{bilayi:ivanoff:kulik:2019}.

\subsection{Convergence of the tail estimators}
Recall 
that the considered tail index estimators are defined by
\begin{align*}
\widehat{\gamma}_{\lfloor nt\rfloor}&\defeq\frac{1}{\sum_{j=1}^{\lfloor nt\rfloor}
\1\{X_j>u_n\}}\sum_{j=1}^{\pe{nt}}\log\left(\frac{X_j}{u_n}\right)\1\{X_j>u_n\} 
\intertext{and}
\widehat{\gamma}_{\rm Hill}(t)&\defeq\frac{1}{\lfloor k_nt\rfloor}\sum\limits_{i=1}^{\lfloor k_nt\rfloor}\log \left(\frac{X_{\lfloor nt \rfloor:\lfloor nt\rfloor-i+1}}{X_{\lfloor nt\rfloor :\lfloor nt\rfloor-k_{\lfloor nt\rfloor}}}\right).
\end{align*}

Based on \ref{thm:TEP} the limiting distributions of $\widehat{\gamma}_{\lfloor nt\rfloor}$ and $\widehat{\gamma}_{\rm Hill}(t)$ can be established in $D[t_0, 1]$ for any $t_0\in (0, 1)$.

\begin{corollary}\label{cor:convergence_gamma_hat}
Let $X_j$, $j\in \mathbb{N}$, be a stationary time series  with marginal tail distribution function $\bar{F}$. Moreover, assume that Assumptions \ref{ass:main} and \ref{ass:technical} hold.
\begin{enumerate}
\item If $\frac{n}{d_{n, r}}=o\left(\sqrt{n\overline{F}(u_n)}\right)$, then
\begin{align*}
\frac{n}{d_{n, r}}t\left(\widehat{\gamma}_{\pe{nt}}-\gamma\right)\Rightarrow 0
\end{align*}
in $D[t_0, 1]$ for all $t_0\in \pr{0,1}$.
\item If $\sqrt{n\overline{F}(u_n)}=o\left(\frac{n}{d_{n, r}}\right)$, then
\begin{align}\label{eq:limit-weakdep}
\sqrt{n\overline{F}(u_n)} t\left(\widehat{\gamma}_{\pe{nt}}-\gamma\right)\Rightarrow  \int_1^{\infty}s^{-1}B\pr{{s^{-\alpha},t}}ds-\alpha^{-1}B\pr{{1,t}}
\end{align}
 in $D[t_0, 1]$ for all $t_0\in \pr{0,1}$.
\end{enumerate}
\end{corollary}

\begin{corollary}\label{cor:convergence_gamma_Hill}
Let $X_j$, $j\in \mathbb{N}$, be a stationary time series  with marginal tail distribution function $\bar{F}$. Moreover, assume that Assumptions \ref{ass:main} and \ref{ass:technical} hold.
\begin{enumerate}
\item If $\frac{n}{d_{n, r}}=o\left(\sqrt{n\overline{F}(u_n)}\right)$, then
\begin{align*}
\frac{n}{d_{n, r}}t\left(\widehat{\gamma}_{\rm Hill}(t)-\gamma\right)\Rightarrow  0
\end{align*}
in $D[t_0, 1]$ for all $t_0\in \pr{0,1}$.
\item If $\sqrt{n\overline{F}(u_n)}=o\left(\frac{n}{d_{n, r}}\right)$, then
\begin{align}\label{eq:limit-weakdep-Hill}
\sqrt{n\overline{F}(u_n)} t\left(\widehat{\gamma}_{\rm Hill}(t)-\gamma\right)\Rightarrow  \int_1^{\infty}s^{-1}B\pr{{s^{-\alpha},t}}ds-\alpha^{-1}B\pr{{1,t}}
\end{align}
in $D[t_0, 1]$ for all $t_0\in \pr{0,1}$.
\end{enumerate}
\end{corollary}

\begin{remark}\label{remark:limit_tail_est}
{\rm
\begin{enumerate}
\item
Following \cite{kulik:soulier:2011}, we conjecture that the proper scaling in the first case is $a_n=\sqrt{n\overline{F}(u_n)}$, as well,  yielding the same limit as in the second case. However, within the scope of this article, we will not consider the corresponding argument in detail.
\item The limit in \eqref{eq:limit-weakdep} and  \eqref{eq:limit-weakdep-Hill} corresponds to $\gamma B(t)$, $t\in [0, 1]$, where $B$ is a standard Brownian motion.
\end{enumerate}
}
\end{remark}

\subsection{Asymptotic distribution of the test statistics}
Recall that the considered test statistics for the change-point problem $\pr{H, A}$ are defined by
\begin{align*}
\Gamma_n&\defeq\sup\limits_{t\in [t_0,1]}t\left|\frac{\widehat{\gamma}_{\lfloor nt\rfloor}}{\widehat{\gamma}_n}-1\right|
\intertext{and}
\widetilde \Gamma_n&\defeq\sup\limits_{t\in [t_0, 1]}t\left|\frac{\widehat{\gamma}_{\rm Hill}(t)}{\widehat{\gamma}_{\rm Hill}(1)}-1\right|\;.
\end{align*}
Using the convergence obtained in Corollaries~~\ref{cor:convergence_gamma_hat} and~ \ref{cor:convergence_gamma_Hill}, we derive the asymptotic
distribution of the test statistics.

\begin{corollary}\label{cor:convergence_test_statistic}
Let $X_j$, $j\in \mathbb{N}$, be a stationary time series  with marginal tail distribution function $\bar{F}$. Moreover, assume that Assumptions \ref{ass:main} and \ref{ass:technical} hold. 
If $\sqrt{n\overline{F}(u_n)}=o\left(\frac{n}{d_{n, r}}\right)$, then, for all $t_0\in \pr{0,1}$, 
\begin{align*}
&\sqrt{n\overline{F}(u_n)}\sup\limits_{t\in [t_0,1]}t\left|\frac{\widehat{\gamma}_{\lfloor nt\rfloor}}{\widehat{\gamma}_n}-1\right|\Rightarrow
\sup\limits_{t\in [t_0, 1]}\left|B(t)-tB(1)\right|,
\end{align*}
where $B(t)$, $t\in [0, 1]$, denotes a standard Brownian motion.
\end{corollary}

\begin{corollary}\label{cor:convergence_test_statistic_Hill}
Let $X_j$, $j\in \mathbb{N}$, be a stationary time series  with marginal tail distribution function $\bar{F}$. Moreover, assume that Assumptions \ref{ass:main} and \ref{ass:technical} hold.
If $\sqrt{n\overline{F}(u_n)}=o\left(\frac{n}{d_{n, r}}\right)$, then, for all $t_0\in \pr{0,1}$, 
\begin{align*}
&\sqrt{n\overline{F}(u_n)}\sup\limits_{t\in [t_0, 1]}t\left|\frac{\widehat{\gamma}_{\rm Hill}(t)}{\widehat{\gamma}_{\rm Hill}(1)}-1\right|\Rightarrow
\sup\limits_{t\in [t_0, 1]}\left|B(t)-tB(1)\right|,
\end{align*}
where $B(t)$, $t\in [0, 1]$, denotes a standard Brownian motion.
\end{corollary}

\section{Simulations}\label{sec:simulations}

For all simulations, the following specifications are made:
\begin{align}
X_j=\sigma(Y_j)\varepsilon_j,\; \ \ j\geq 1\;,\label{eq: LMSVsimulation_normal}
\end{align}
where
\begin{itemize}
\item  $\varepsilon_j$, $j\geq 1$, is an independent, identically distributed sequence of Pareto distributed  random variables  generated  by  the function \verb$rgpd$  (\verb$fExtremes$ package in \verb$R$);
\item $Y_j$, $j\geq 1$, is a fractional Gaussian noise sequence generated by the function \verb$simFGN0$ (\verb$longmemo$ package in \verb$R$) with  Hurst parameter $H$;
\item  $\sigma(y)=\exp(y)$.
\end{itemize}

Under the alternative, we insert a change of height $h$ at location $k=\lfloor n\tau\rfloor$ by simulating  independent, identically   Pareto distributed observations $\varepsilon_j$, $j\geq 1$, with $\varepsilon_j$, $j=1, \ldots, k$, having tail index $\alpha_1=\ldots =\alpha_k=\alpha$ and $\varepsilon_j$, $j=k+1, \ldots, n$, having tail index $\alpha_{k+1}=\ldots =\alpha_n=\alpha+h$.

We base test decisions on the statistic $\widetilde\Gamma_{n}\defeq \max\limits_{1\leq k\leq n-1}\Gamma_{k, n}$, where
\begin{align}\label{eq:test-stat}
\widetilde\Gamma_{k, n}=\frac{k}{n}\left|\frac{\widehat{\gamma}_{\rm Hill}\left(\frac{k}{n}\right)}{\widehat{\gamma}_{\rm Hill}(1)}-1\right| \text{ with }
\widehat{\gamma}_{\rm Hill}\left(t\right)=\frac{1}{\lfloor k_nt\rfloor}\sum\limits_{i=1}^{\lfloor k_nt\rfloor}\log \left(\frac{X_{\lfloor nt\rfloor:\lfloor nt\rfloor-i+1}}{X_{\lfloor nt\rfloor :\lfloor nt\rfloor-\lfloor k_nt\rfloor}}\right),
\end{align}
and we choose a significance level of $5\%$.

For the computation of the test statistic, the choice of $k_n$, i.e., the number of largest observations that contribute to the estimation of the tail index,
is  considered a delicate issue.
In fact, it has been shown in \cite{hall:1982} that the optimal choice
of $k_n$ depends on the tail behavior of the data-generating process. 
Due to this circularity, \cite{dumouchel:1983} suggests to chose $k_n$ proportional to the sample size. 
As noted in \cite{quintos:2001}, a corresponding choice of $k_n$ has been shown to perform well in simulations and is widely used by practitioners.
Hence, we choose $k_n=\lfloor np\rfloor$, i.e.,  $p$ defines the proportion of the data that the estimation of the tail index is based on.

The power of the testing procedures is analyzed by considering different choices for the height of the level shift, denoted by $h$, and the location of the change-point, denoted by $\tau$. In the tables, the columns  that are superscribed by $h=0$ correspond to the frequency of a type 1 error, i.e., the rejection rate under the hypothesis.

Considering all simulation results, the first thing to note is that these
concur with the expected behavior of change-point tests:
An increasing sample size goes along with an improvement of the finite sample performance, i.e., the empirical size approaches the level of significance and the empirical power increases, the empirical power of the testing procedures increases when the height of the level shift increases, and the empirical power is higher for breakpoints located in the middle of the sample than for change-point locations that lie close to the boundary of the testing region.

Both Hurst parameter and  tail index, seem to have a significant effect on the rejection rates of the change-point test.
An increase in dependence, i.e., an increase of the Hurst parameter $H$,  leads to an increase in the  number of rejections. On the one hand, this leads to an increase of power, on the other hand, it results in  a larger deviation of the empirical size from the significance level.
An increase of tail thickness, i.e., a decrease of the tail parameter $\alpha$, however, results in an improvement of the test's performance in that the empirical power increases while the empirical size draws closer to the level of significance. Moreover, the empirical power of the test is higher for changes to heavier tails, i.e., the test tends to detect changes with a negative change-point height $h$ better.

Technically speaking, the particular case of a change with height $h=-1$ from $\alpha=1$
to $\alpha=0$  does not fall under our model assumptions.
For a change after a proportion $\tau=0.5$ of the data, the empirical power is extremely low in this case. However, for an early change, i.e., for $\tau=0.25$, the empirical power is comparatively high.  
\newpage

\begin{landscape}
\thispagestyle{empty}
\begin{table}
\begin{threeparttable}
\scriptsize
\begin{tabular}{ccccccccccccccccccccc}
 &  &  &  &  \multicolumn{5}{c}{$\alpha=2.5$}  &  & \multicolumn{5}{c}{$\alpha=2$}  &  & \multicolumn{5}{c}{$\alpha=1$} \\
  \cline{5-9}   \cline{11-15}   \cline{17-21}\\
 & $p$ & $n$ &  & $h=0$ & $h=0.5$ & $h=1$ & $h=-0.5$ & $h=-1$ &  & $h=0$ & $h=0.5$ & $h=1$ & $h=-0.5$ & $h=-1$ &  & $h=0$ & $h=0.5$ & $h=1$ & $h=-0.5$ & $h=-1$ \\
 \cline{2-21}
 \\
 \multirow{6}{*}{\rotatebox{90}{$H=0.6$}}
& 0.1 & 300 &  & 0.088 & 0.088 & 0.086 & 0.109 & 0.192 &  & 0.097 & 0.086 & 0.086 & 0.145 & 0.418 &  & 0.100 & 0.110 & 0.128 & 0.641 & 0.056 \\
 &  & 500  &  & 0.078 & 0.069 & 0.065 & 0.105 & 0.249 &  & 0.078 & 0.083 & 0.069 & 0.148 & 0.602 &  & 0.092 & 0.129 & 0.141 & 0.842 & 0.053 \\
 &  & 1000 &  & 0.071 & 0.063 & 0.059 & 0.106 & 0.391 &  & 0.062 & 0.072 & 0.073 & 0.195 & 0.853 &  & 0.077 & 0.188 & 0.222 & 0.979 & 0.054 \\
 & 0.2 & 300 &  & 0.071 & 0.065 & 0.058 & 0.078 & 0.176 &  & 0.072 & 0.073 & 0.071 & 0.112 & 0.485 &  & 0.076 & 0.125 & 0.178 & 0.816 & 0.095 \\
 &  & 500  &  & 0.049 & 0.059 & 0.059 & 0.076 & 0.227 &  & 0.060 & 0.061 & 0.071 & 0.123 & 0.687 &  & 0.067 & 0.187 & 0.249 & 0.951 & 0.133 \\
 &  & 1000  &  & 0.044 & 0.050 & 0.055 & 0.086 & 0.387 &  & 0.047 & 0.053 & 0.075 & 0.185 & 0.911 &  & 0.062 & 0.308 & 0.428 & 0.999 & 0.235 \\
 &  & &  & &  &  & & &  &  &  &  &  & &  & & &  &  & \\
 \multirow{6}{*}{\rotatebox{90}{$H=0.7$}} & 0.1 & 300 &  & 0.112 & 0.103 & 0.096 & 0.137 & 0.217 &  & 0.114 & 0.100 & 0.102 & 0.159 & 0.443 &  & 0.106 & 0.130 & 0.131 & 0.642 & 0.055 \\
 &  & 500 &  & 0.093 & 0.086 & 0.087 & 0.123 & 0.262 &  & 0.096 & 0.082 & 0.091 & 0.166 & 0.626 &  & 0.095 & 0.131 & 0.148 & 0.835 & 0.052 \\
&  & 1000 &  & 0.084 & 0.069 & 0.070 & 0.118 & 0.385 &  & 0.077 & 0.077 & 0.086 & 0.213 & 0.866 &  & 0.084 & 0.199 & 0.239 & 0.979 & 0.062 \\
 & 0.2 & 300 &  & 0.087 & 0.083 & 0.083 & 0.105 & 0.196  &  & 0.080 & 0.085 & 0.090 & 0.131 & 0.502 &  & 0.084 & 0.134 & 0.184 & 0.824 & 0.092 \\
&  & 500 &  & 0.075 & 0.080 & 0.071 & 0.099 & 0.256 &  & 0.073 & 0.073 & 0.081 & 0.147 & 0.702 &  & 0.075 & 0.196 & 0.263 & 0.959 & 0.122 \\
&  & 1000 &  & 0.068 & 0.063 & 0.067 & 0.109 & 0.408 &  & 0.068 & 0.077 & 0.090 & 0.211 & 0.919 &  & 0.066 & 0.317 & 0.458 & 0.999 & 0.235 \\
 &  &  &  &   & & &  & &  &  &  &   &  & &  &   &  &  &  &\\
 \multirow{6}{*}{\rotatebox{90}{$H=0.8$}} & 0.1 & 300&  & 0.140 & 0.125 & 0.124 & 0.149 & 0.248 &  & 0.130 & 0.135 & 0.124 & 0.186 & 0.477 &  & 0.116 & 0.132 & 0.148 & 0.637 & 0.053 \\
 &  & 500 &  & 0.131 & 0.122 & 0.113 & 0.156 & 0.313 &  & 0.108 & 0.119 & 0.112 & 0.197 & 0.662 &  & 0.101 & 0.146 & 0.171 & 0.842 & 0.053 \\
 &  & 1000 &  & 0.108 & 0.109 & 0.108 & 0.167 & 0.446 &  & 0.107 & 0.113 & 0.115 & 0.254 & 0.881 &  & 0.090 & 0.217 & 0.264 & 0.984 & 0.058 \\
 & 0.2 & 300 &  & 0.118 & 0.110 & 0.116 & 0.135 & 0.250 &  & 0.107 & 0.112 & 0.118 & 0.174 & 0.560 &  & 0.092 & 0.171 & 0.217 & 0.837 & 0.078 \\
 &  & 500 &  & 0.122 & 0.111 & 0.103 & 0.157 & 0.319 &  & 0.101 & 0.109 & 0.118 & 0.192 & 0.743 &  & 0.079 & 0.209 & 0.299 & 0.957 & 0.123 \\
 &  & 1000 &  & 0.098 & 0.107 & 0.113 & 0.165 & 0.491 &  & 0.100 & 0.117 & 0.142 & 0.269 & 0.935 &  & 0.076 & 0.360 & 0.493 & 0.999 & 0.215 \\
 &  &  &  &  &  &  & &&  & & &  &
 & &  & & &  &  &   \\
 \multirow{6}{*}{\rotatebox{90}{$H=0.9$}} & 0.1 & 300 &  & 0.175 & 0.164 & 0.165 & 0.192 & 0.308 &  & 0.166 & 0.151 & 0.162 & 0.215 & 0.530 &  & 0.120 & 0.152 & 0.167 & 0.650 & 0.059 \\
&  & 500 &  & 0.167 & 0.165 & 0.164 & 0.201 & 0.395 &  & 0.152 & 0.151 & 0.159 & 0.244 & 0.715 &  & 0.105 & 0.166 & 0.198 & 0.834 & 0.053 \\
 &  & 1000  &  & 0.175 & 0.171 & 0.176 & 0.247 & 0.554 &  & 0.156 & 0.166 & 0.185 & 0.322 & 0.924 &  & 0.104 & 0.239 & 0.289 & 0.982 & 0.056 \\
 & 0.2 & 300 &  & 0.169 & 0.158 & 0.168 & 0.194 & 0.341 &  & 0.140 & 0.162 & 0.161 & 0.215 & 0.646 &  & 0.102 & 0.200 & 0.268 & 0.846 & 0.063 \\
&  & 500 &  & 0.177 & 0.183 & 0.171 & 0.213 & 0.458 &  & 0.153 & 0.158 & 0.180 & 0.275 & 0.821 &  & 0.101 & 0.262 & 0.373 & 0.964 & 0.079 \\
&  & 1000 &  & 0.207 & 0.203 & 0.215 & 0.281 & 0.625 &  & 0.175 & 0.192 & 0.230 & 0.372 & 0.966 &  & 0.100 & 0.414 & 0.557 & 0.999 & 0.154 \\
\end{tabular}
\caption{\footnotesize Rejection rates of the change-point test based on the statistic $\Gamma_n$, $k_n=\lfloor np\rfloor$,   for  LMSV  time series (Pareto distributed $\varepsilon_j,$ $j\geq 1$)   of length $n$ with Hurst parameter $H$, tail index $\alpha$  and a shift in the mean of height $h$ after a proportion $\tau=0.5$. The calculations are based on 5,000 simulation runs.}
\end{threeparttable}
\end{table}
\end{landscape}

\begin{landscape}
\thispagestyle{empty}
\begin{table}
\begin{threeparttable}
\scriptsize
\begin{tabular}{ccccccccccccccccccccc}
 &  &  &  &  \multicolumn{5}{c}{$\alpha=2.5$}  &  & \multicolumn{5}{c}{$\alpha=2$}  &  & \multicolumn{5}{c}{$\alpha=1$} \\
  \cline{5-9}   \cline{11-15}   \cline{17-21}\\
 & $p$ & $n$ &  & $h=0$ & $h=0.5$ & $h=1$ & $h=-0.5$ & $h=-1$ &  & $h=0$ & $h=0.5$ & $h=1$ & $h=-0.5$ & $h=-1$ &  & $h=0$ & $h=0.5$ & $h=1$ & $h=-0.5$ & $h=-1$ \\
 \cline{2-21}
 \\
 \multirow{6}{*}{\rotatebox{90}{$H=0.6$}} & 0.1 & 300 &  & 0.088 & 0.086 & 0.085 & 0.104 & 0.127 &  & 0.097 & 0.099 & 0.092 & 0.105 & 0.145 &  & 0.100 & 0.148 & 0.198 & 0.183 & 0.069 \\
 &  & 500 &  & 0.078 & 0.071 & 0.071 & 0.083 & 0.129 &  & 0.078 & 0.072 & 0.075 & 0.105 & 0.203 &  & 0.092 & 0.155 & 0.238 & 0.254 & 0.078 \\
 &  & 1000 &  & 0.071 & 0.058 & 0.060 & 0.076 & 0.151 &  & 0.062 & 0.061 & 0.075 & 0.106 & 0.373 &  & 0.077 & 0.216 & 0.376 & 0.594 & 0.137 \\
 & 0.2 & 300 &  & 0.071 & 0.069 & 0.068 & 0.075 & 0.099 &  & 0.072 & 0.074 & 0.073 & 0.089 & 0.156 &  & 0.076 & 0.149 & 0.230 & 0.272 & 0.658 \\
 &  & 500 &  & 0.049 & 0.052 & 0.059 & 0.063 & 0.120 &  & 0.060 & 0.062 & 0.070 & 0.084 & 0.262 &  & 0.067 & 0.185 & 0.328 & 0.532 & 0.891 \\
 &  & 1000 &  & 0.044 & 0.050 & 0.052 & 0.056 & 0.160 &  & 0.047 & 0.055 & 0.072 & 0.096 & 0.521 &  & 0.062 & 0.295 & 0.550 & 0.912 & 0.997 \\
 &  & &  && & &  & &  &  &  & &&  &  &  & & &  &  \\
 \multirow{6}{*}{\rotatebox{90}{$H=0.7$}} & 0.1 & 300 &  & 0.112 & 0.100 & 0.110 & 0.124 & 0.139 &  & 0.114 & 0.102 & 0.104 & 0.125 & 0.168 &  & 0.106 & 0.146 & 0.191 & 0.176 & 0.061 \\
 &  & 500 &  & 0.093 & 0.091 & 0.092 & 0.100 & 0.145 &  & 0.096 & 0.096 & 0.091 & 0.110 & 0.206 &  & 0.095 & 0.178 & 0.245 & 0.251 & 0.082 \\
&  & 1000 &  & 0.084 & 0.074 & 0.075 & 0.092 & 0.176 &  & 0.077 & 0.076 & 0.091 & 0.116 & 0.393 &  & 0.084 & 0.236 & 0.388 & 0.591 & 0.122 \\
 & 0.2 & 300 & &  0.0868 & 0.081 & 0.073 & 0.087 & 0.113 &  & 0.080 & 0.085 & 0.097 & 0.100 & 0.185 &  & 0.084 & 0.148 & 0.246 & 0.297 & 0.653 \\
&  & 500 &  & 0.075 & 0.071 & 0.076 & 0.080 & 0.122 &  & 0.073 & 0.077 & 0.082 & 0.104 & 0.285 &  & 0.075 & 0.206 & 0.343 & 0.532 & 0.879 \\
 &  & 1000 &  & 0.068 & 0.068 & 0.073 & 0.084 & 0.187 &  & 0.068 & 0.066 & 0.088 & 0.114 & 0.571 &  & 0.066 & 0.305 & 0.567 & 0.922 & 0.994 \\
 &  & &  &&  & && &  &  & &  &  & &  & & &  &  &  \\
 \multirow{6}{*}{\rotatebox{90}{$H=0.8$}} & 0.1 & 300 &  & 0.140 & 0.135 & 0.132 & 0.136 & 0.152 &  & 0.130 & 0.141 & 0.126 & 0.134 & 0.186 &  & 0.116 & 0.164 & 0.211 & 0.182 & 0.062 \\
 &  & 500 &  & 0.131 & 0.122 & 0.126 & 0.130 & 0.166 &  & 0.108 & 0.123 & 0.137 & 0.135 & 0.216 &  & 0.101 & 0.185 & 0.283 & 0.251 & 0.073 \\
 &  & 1000 &  & 0.108 & 0.117 & 0.115 & 0.127 & 0.220 &  & 0.107 & 0.108 & 0.128 & 0.145 & 0.434 &  & 0.090 & 0.266 & 0.420 & 0.599 & 0.113 \\
 & 0.2 & 300 &  & 0.118 & 0.119 & 0.121 & 0.123 & 0.149 &  & 0.107 & 0.119 & 0.109 & 0.125 & 0.203 &  & 0.092 & 0.177 & 0.261 & 0.300 & 0.619 \\
 &  & 500 &  & 0.122 & 0.107 & 0.117 & 0.123 & 0.173 &  & 0.101 & 0.111 & 0.121 & 0.135 & 0.326 &  & 0.079 & 0.227 & 0.370 & 0.540 & 0.851 \\
 &  & 1000 &  & 0.098 & 0.113 & 0.113 & 0.137 & 0.259 &  & 0.100 & 0.111 & 0.129 & 0.157 & 0.625 &  & 0.076 & 0.332 & 0.588 & 0.922 & 0.994 \\
 &  & &  &  &  & & &  &  &  & & & &  &  & &  &  &  &  \\
 \multirow{6}{*}{\rotatebox{90}{$H=0.9$}} & 0.1 & 300 &  & 0.175 & 0.181 & 0.187 & 0.169 & 0.173 &  & 0.166 & 0.165 & 0.177 & 0.168 & 0.192 &  & 0.120 & 0.193 & 0.268 & 0.176 & 0.054 \\
 &  & 500 &  & 0.167 & 0.181 & 0.180 & 0.167 & 0.204 &  & 0.152 & 0.164 & 0.175 & 0.160 & 0.252 &  & 0.105 & 0.212 & 0.324 & 0.266 & 0.056 \\
 &  & 1000 &  & 0.175 & 0.180 & 0.192 & 0.195 & 0.289 &  & 0.156 & 0.170 & 0.200 & 0.202 & 0.509 &  & 0.104 & 0.295 & 0.492 & 0.602 & 0.080 \\
 & 0.2 & 300 &  & 0.169 & 0.171 & 0.174 & 0.166 & 0.184 &  & 0.140 & 0.155 & 0.179 & 0.163 & 0.261 &  & 0.102 & 0.206 & 0.349 & 0.304 & 0.501 \\
&  & 500 &  & 0.177 & 0.175 & 0.197 & 0.183 & 0.252 &  & 0.153 & 0.164 & 0.197 & 0.182 & 0.414 &  & 0.101 & 0.259 & 0.455 & 0.580 & 0.759 \\
 &  & 1000 &  & 0.207 & 0.215 & 0.229 & 0.236 & 0.377 &  & 0.175 & 0.200 & 0.222 & 0.243 & 0.755 &  & 0.100 & 0.412 & 0.684 & 0.941 & 0.966 \\
\end{tabular}
\caption{\footnotesize Rejection rates of the change-point test based on the statistic $\Gamma_n$, $k_n=\lfloor np\rfloor$,   for  LMSV  time series (Pareto distributed $\varepsilon_j,$ $j\geq 1$)   of length $n$ with Hurst parameter $H$, tail index $\alpha$  and a shift in the mean of height $h$ after a proportion $\tau=0.25$. The calculations are based on 5,000 simulation runs.}
\label{rejection rates tau=0.25}
\end{threeparttable}
\end{table}
\end{landscape}

\newpage

\section{Data}\label{sec:data}

The analysis of financial time series, such as stock market prices, usually focuses on log returns instead of the observed data itself. As an example, we consider the log returns of the daily closing indices  of Standard \& Poor's 500 (S{\&}P 500, in short)
defined by
\begin{align*}
L_t\defeq\log R_t, \ R_t\defeq\frac{P_t}{P_{t-1}},
\end{align*}
where $P_t$ denotes the value of the index on day $t$, in the
period from January 2008  to December 2008; see Figure \ref{S&P_500}.

\begin{figure}[htbp]
\begin{center}
\includegraphics[scale=0.6]{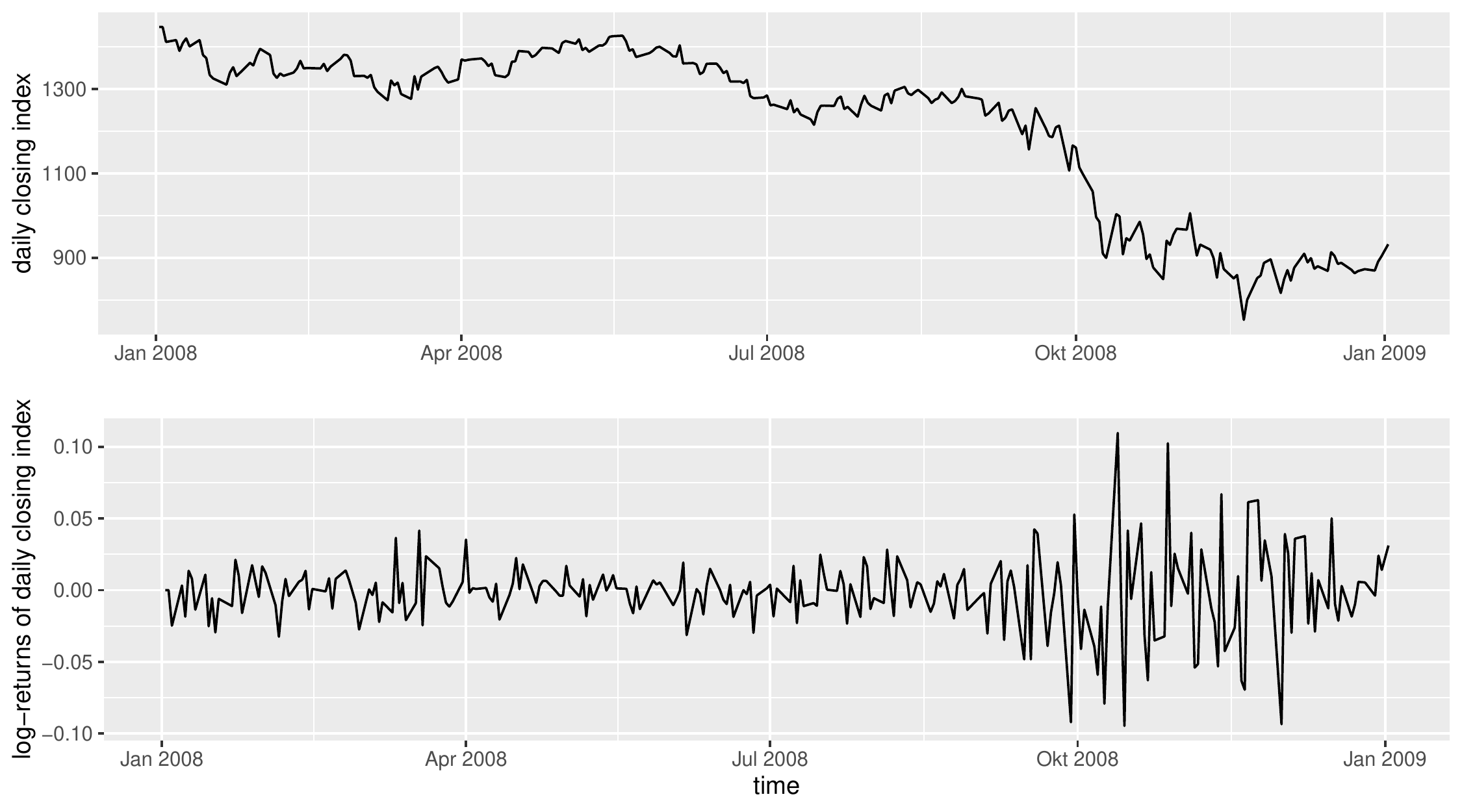}
\end{center}
\caption{Daily closing index of Standard \& Poor's 500 and its log returns
 from January 2008  to December 2008.
The data has been obtained from Google Finance.}
\label{S&P_500}
\end{figure}

Comparing the plots of the sample autocorrelation function
of the log returns and the sample autocorrelation function of their absolute values in Figure \ref{S&P_500_acf}, we observe  a phenomenon that is  often encountered in the context of financial data:
the log returns of the index appear to be uncorrelated, whereas  the absolute log returns  tend to be  highly correlated.

\begin{figure}[htbp]
\begin{center}
\includegraphics[scale=0.6]{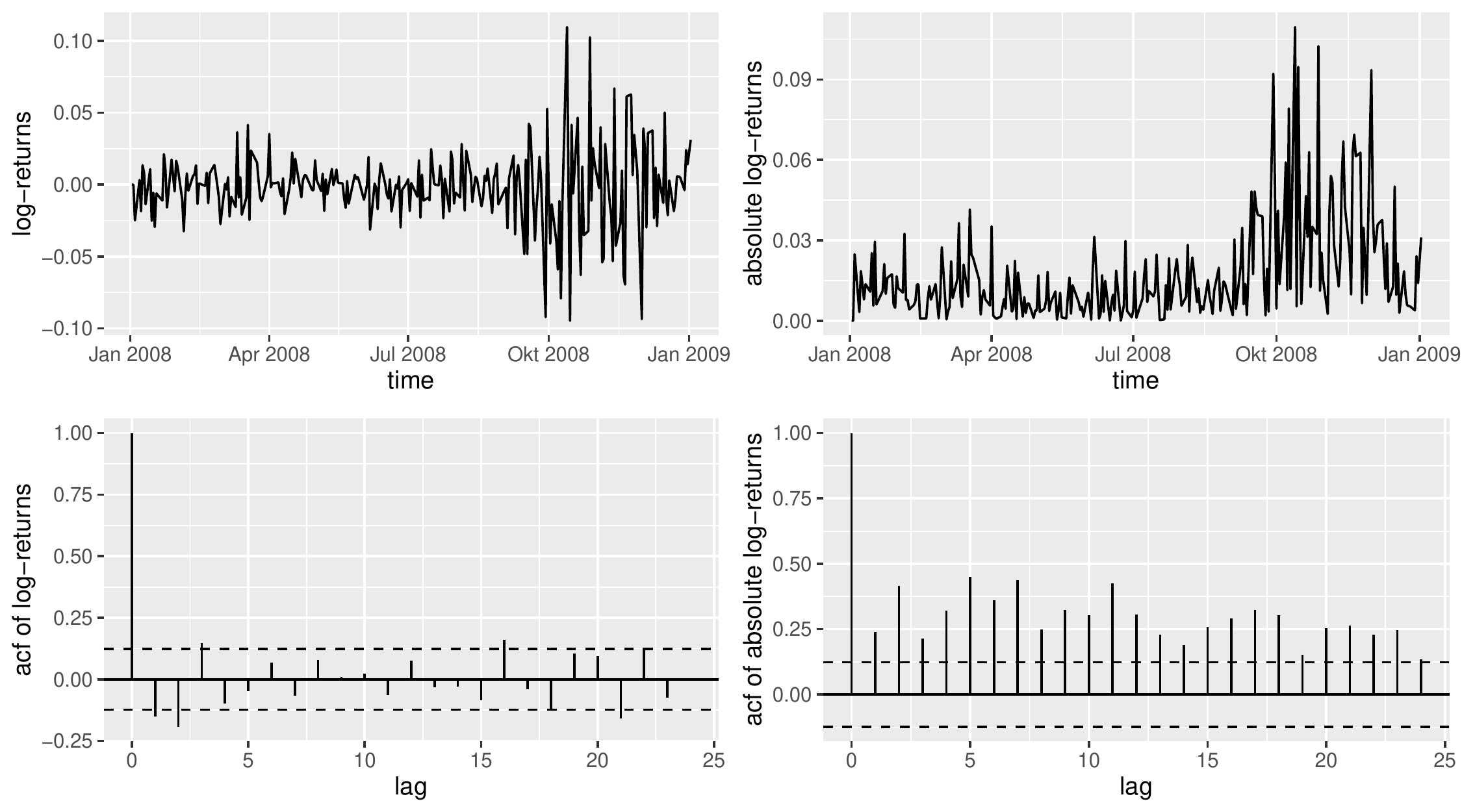}
\end{center}
\caption{Sample autocorrelation of the log returns and the absolute log returns of Standard \& Poor's 500  daily closing index from January 2005 to December 2010.
The two dashed horizontal lines mark the bounds for the 95\% confidence interval of the autocovariances under the assumption of data generated by white noise.
}
\label{S&P_500_acf}
\end{figure}

Moreover, the plot in Figure \ref{S&P_500} shows that the considered time series exhibits {\em volatility clustering}, meaning that large price changes, i.e., log returns with relatively large absolute values, tend to cluster. This  indicates that  observations are not independent across time, although the absence of linear autocorrelation
suggests that the dependence is nonlinear; see \cite{cont:2005}.

Another characteristic of financial time series is  the occurrence of heavy tails.
In particular,  probability distributions of  log returns
often exhibit
tails which are heavier than those of a normal distribution.
For the S{\&}P 500 data, this property is highlighted by
the Q-Q plot in Figure \ref{S&P_500_qqplot}.
\begin{figure}[htbp]
\begin{center}
\includegraphics[scale=0.6]{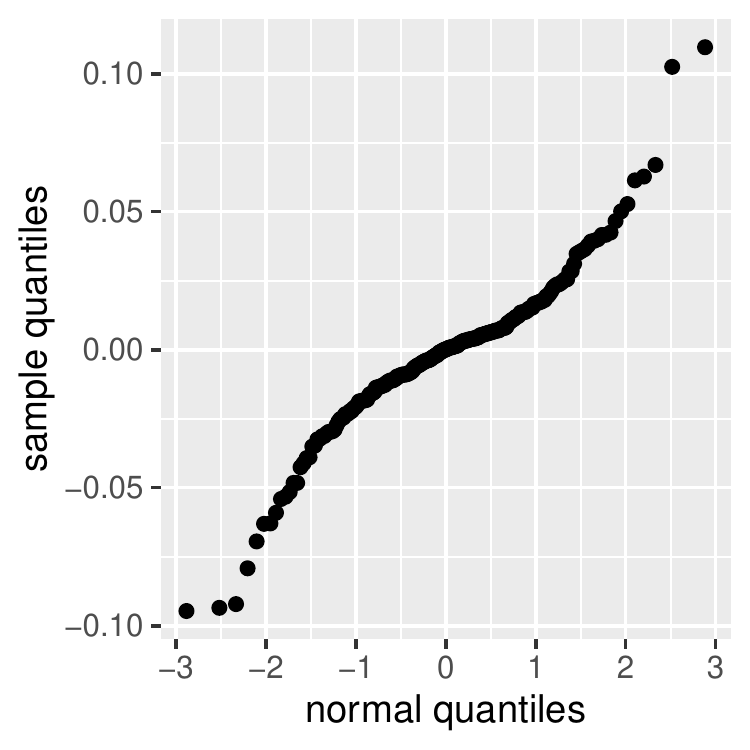}
\end{center}
\caption{Q-Q plot for the log returns of Standard \& Poor's 500  daily closing index from January 2005  to December 2010.}
\label{S&P_500_qqplot}
\end{figure}

All of the previously described features of financial data can be covered by  the LMSV model considered in our paper.

In view of the fact that the LMSV model  captures properties of the log returns of Standard \& Poor's 500 daily closing index, we are interested in analyzing the data with respect to a change in the tail index.

As in our simulations, we base the test decision on the statistic
defined in \eqref{eq:test-stat}.
We choose $k_n=\lfloor np\rfloor$, i.e.,  $p$ defines the proportion of the data that the estimation of the tail index is based on.
Choosing $p=0.1$, the value of the test statistic corresponds to $\widetilde\Gamma_n=1.467503$.
The  95\%-quantile of the limit distribution $\sup_{t\in [0,1]}
\left|B(t)-tB(1)\right|$ equals $1.3463348$. Choosing the critical value for the hypothesis test correspondingly, the value of $\Gamma_n$ therefore indicates a change-point in the tail index at a level of significance of 5\%.

A natural estimate for the change-point location is given by that point in time $k$, where $\Gamma_{k, n}$ attains its maximum.
For the considered data, this point in time corresponds  to September 16, 2008, i.e., one day after September 15, 2008, the day Lehman Brothers   filed for bankruptcy protection; see Figure \ref{S&P_500_with_cp}.

\begin{figure}[htbp]
\begin{center}
\includegraphics[scale=0.6]{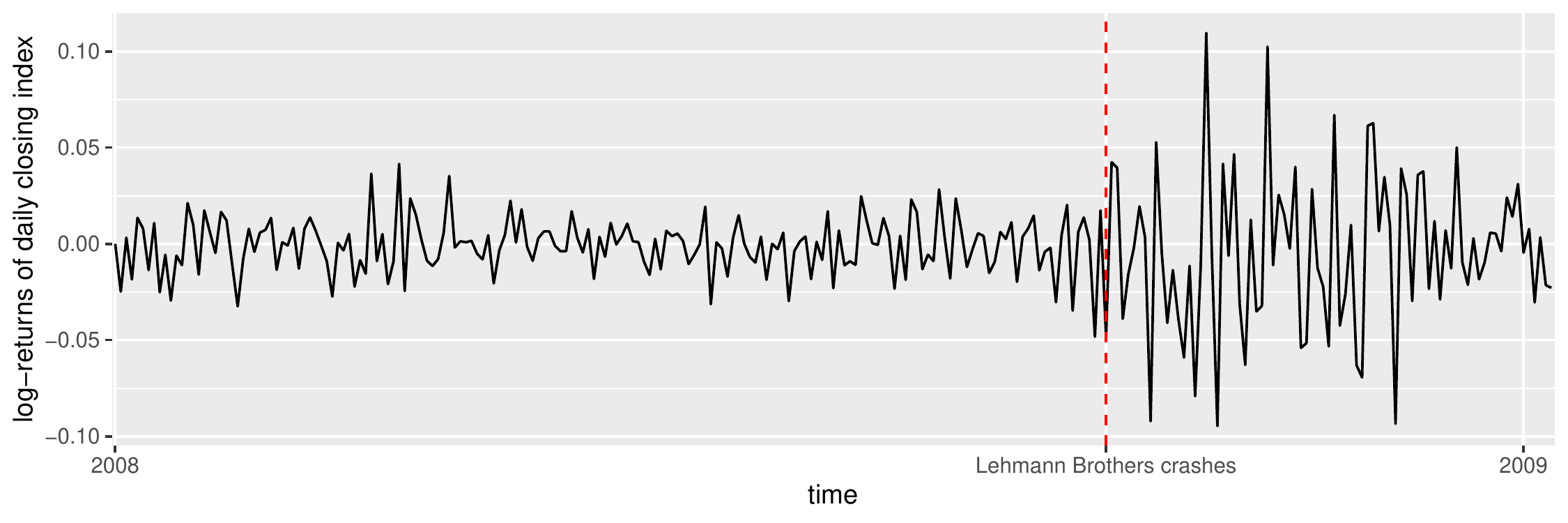}
\end{center}
\caption{Log returns of the daily closing index of Standard \& Poor's 500
 from January 2008  to December 2008.
The red dashed line indicates the estimated change-point location.}
\label{S&P_500_with_cp}
\end{figure}

\section{Proofs}\label{sec:proofs}

\subsection{Proof of Theorem~\ref{thm:TEP}}

\subsubsection{Decomposition of the tail empirical process}\label{subsubsec:decomposition_empirical_process}

Recall that
\begin{equation*}
e_n(s,t)=\left\{\widetilde{T}_n(s,t)-T(s,t)\right\},
\end{equation*}
where
\begin{equation*}
\widetilde{T}_n(s,t)\defeq\frac{1}{n\bar{F}(u_n)}\sum\limits_{j=1}^{\pe{nt}}\1\left\{X_j> u_ns\right\} \ \text{ and } \ T\pr{s,t}=ts^{-\alpha}.
\end{equation*}
To prove Theorem~\ref{thm:TEP}, we consider the following decomposition:
\begin{align*}
&e_n(s, t)
=\left\{\widetilde{T}_n(s, t)-T_n(s, t)\right\}+\left\{T_n(s, t)-T(s, t)\right\},
\intertext{where}
&T_n(s, t)\defeq\Es{\widetilde{T}_n(s, t)}=\frac{\lfloor nt\rfloor}{n}\frac{\overline{F}(u_ns)}{\overline{F}(u_n)}.
\end{align*}
Obviously, it holds that
$$\lim\limits_{n\rightarrow \infty}T_n(s, t)=T(s, t)$$
for $s>0$ and $t\in \left[0, 1\right]$. In particular, the convergence holds  uniformly on compact subsets of $(0,\infty)\times [0,1]$. Moreover,
for any $s_0>0$ it holds that
\begin{align*}
 \sup\limits_{s\geq s_0, \ t\in [0, 1]}\left|T_n(s, t)-T(s, t)\right|
&\leq \sup\limits_{s\geq s_0}\frac{\overline{F}(u_ns)}{\overline{F}(u_n)} \sup\limits_{t\in [0, 1]}\left|\frac{\lfloor nt\rfloor}{n}-t\right|+\sup\limits_{s\geq s_0}\left|\frac{\overline{F}(u_ns)}{\overline{F}(u_n)}-s^{-\alpha}\right|.
\end{align*}

Note that
\begin{align*}
 \sup\limits_{t\in [0, 1]}\left|\frac{\lfloor nt\rfloor}{n}-t\right|=o\pr{\frac{d_{n, r}}{n}+\frac{1}{\sqrt{n\overline{F}(u_n)}}}.
\end{align*}

Due to Proposition 2.8 in \cite{kulik:soulier:2011} and \ref{assump:eta_of_un}
we have $$
\sup_{s\geq s_0}\left|\frac{\overline{F}(u_ns)}{\overline{F}(u_n)}-s^{-\alpha}\right|=o\pr{\frac{d_{n, r}}{n}+\frac{1}{\sqrt{n\overline{F}(u_n)}}},
 $$
 which implies
 $$
\sup_{s\geq s_0, \ t\in [0, 1]}\left|T_n(s, t)-T(s, t)\right|=o\pr{\frac{d_{n, r}}{n}+\frac{1}{\sqrt{n\overline{F}(u_n)}}}.
$$
Since 
\begin{align*}
\lim_{n\to\infty}
\frac{\overline{F}_{\varepsilon}(u_n)}{\overline{F}(u_n)}=
\frac{1}{\Es{\sigma^{\alpha}(Y_1)}}
\end{align*}
by \ref{assump:moments:isigmaY} and Breiman's Lemma,
it therefore suffices to study weak convergence of the process
\begin{align*}
\widetilde{e}_n(s,t)
&=\frac{1}{n\overline{F}_{\varepsilon}(u_n)}\sum\limits_{j=1}^{\lfloor nt\rfloor}\left(\1\left\{X_j>u_ns\right\} -\overline{F}(u_ns)\right).
\end{align*}
For this, we consider the following decomposition:
\begin{align}\label{eq:decomposition}
&\widetilde{e}_n(s, t)\eqdef M_n(s, t)+R_n(s, t),
\end{align}
where
\begin{align*}
M_n(s, t)\defeq &\frac{1}{n\overline{F}_{\varepsilon}(u_n)}\sum\limits_{j=1}^{\lfloor nt\rfloor}\left(\1\left\{X_j>u_ns\right\} -
\Es{\1\left\{X_j>u_n s\right\} \left|\right. \mathcal{F}_{j-1} }\right),\\
R_n(s, t)\defeq &\frac{1}{n\overline{F}_{\varepsilon}(u_n)}\sum\limits_{j=1}^{\lfloor nt\rfloor}\left(
\Es{\1\left\{X_j>u_n s\right\} \left|\right. \mathcal{F}_{j-1}\right)-\overline{F}(u_ns)},
\end{align*}
and
\begin{equation}\label{eq:definition_of_Fj}
\mathcal F_j\defeq\sigma\pr{\varepsilon_k,\eta_k,k\in\Z,,k\leq j}.
\end{equation}
We call $M_n$ the \textit{martingale part}, while we refer to $R_n$ as the \textit{long-range dependent  part}.

\subsubsection{The long-range dependent  part}\label{sec:proof-lrd}
\begin{proposition}[Weak convergence of $R_n(s, t)$]\label{lem:weak_conv_of_LRD_part}
Under the assumptions of Theorem \ref{thm:TEP}, the following holds:
\begin{align*}
\frac{n}{d_{n, r}}R_n(s, t)\Rightarrow  s^{-\alpha}\frac{1}{r!}J_r(\Psi)Z_{r, H}(t),
\end{align*}
where $\Rightarrow$ denotes weak convergence in $D\left([1, \infty]\times
[0, 1]\right)$.
\end{proposition}

\begin{proof}
Note that
\begin{equation}\label{eq:conditional_expectation_martingale_part}
\Es{\1 \ens{X_j>u_ns }\left|\right.\mathcal{F}_{j-1}}=\overline{F}_{\varepsilon}\left(\frac{u_n s}{\sigma(Y_j)}\right)
\end{equation}
and 
\begin{align*}
\Es{\overline{F}_{\varepsilon}\left(\frac{u_ns}{\sigma(Y_j)}\right)}=\Es{\Es{\1\left\{X_j>u_n s\right\} \left|\right. \mathcal{F}_{j-1}}}
=\Es{\1\ens{X_j>u_n s }}=\overline{F}(u_ns).
\end{align*}
As a result, we can rewrite $R_n(s, t)$ as follows:
\begin{align*}
R_n(s, t)
&=\frac{1}{n\overline{F}_{\varepsilon}(u_n)}\sum\limits_{j=1}^{\lfloor nt\rfloor}\left(\overline{F}_{\varepsilon}\left(\frac{u_ns}{\sigma(Y_j)}\right)-
\Es{\overline{F}_{\varepsilon}\left(\frac{u_ns}{\sigma(Y_j)}\right)}\right)
=\frac{1}{n}\sum\limits_{j=1}^{\lfloor nt\rfloor}\left(\Psi_n(Y_j, s)-\Es{\Psi_n(Y_j, s)}\right),
\end{align*}
where
\begin{align}\label{eq:psin}
\Psi_n(y, s)=\frac{\overline{F}_{\varepsilon}\left(\frac{u_ns}{\sigma(y)}\right)}{\overline{F}_{\varepsilon}(u_n)}.
\end{align}
Due to regular variation of $\overline{F}_{\varepsilon}$, we have
\begin{align}\label{eq:psis}
\Psi_n(y, s)=\frac{\overline{F}_{\varepsilon}\left(\frac{u_ns}{\sigma(y)}\right)}{\overline{F}_{\varepsilon}(u_n)}
\sim
\left(\frac{s}{\sigma(y)}\right)^{-\alpha}=s^{-\alpha}\Psi(y).
\end{align}
Furthermore, it holds that
\begin{align}
R_n(s, t)
=&\frac{1}{n}\sum\limits_{j=1}^{\lfloor nt\rfloor}\left(\Psi_n(Y_j, s)-\Es{\Psi_n(Y_j, s)}\right)\notag\\
=&\frac{1}{n}\sum\limits_{j=1}^{\lfloor nt\rfloor}\left(\Psi_s(Y_j)-\Es{\Psi_s(Y_j)}\right) +\frac{1}{n}\sum\limits_{j=1}^{\lfloor nt\rfloor}\left(\Psi_n(Y_j, s)-\Psi_s(Y_j)\right)\notag\\
&+\frac{1}{n}\sum\limits_{j=1}^{\lfloor nt\rfloor}\left(\Es{\Psi_s(Y_j)}-\Es{\Psi_n(Y_j, s)}\right),\label{eq:Rn-summands}
\end{align}
where $\Psi_s(y)\defeq s^{-\alpha}\Psi(y)$.

As $\Es{\sigma^{2\alpha}(Y_1)}<\infty$ by \ref{assump:moments:isigmaY}, the functional non-central limit theorem of \cite{taqqu:1979} yields
\begin{align*}
\frac{n}{d_{n, r}}\frac{1}{n}\sum\limits_{j=1}^{\lfloor nt\rfloor}\left(\Psi_s(Y_j)-
\Es{\Psi_s(Y_j)}\right)=s^{-\alpha}\frac{n}{d_{n, r}}\frac{1}{n}\sum\limits_{j=1}^{\lfloor nt\rfloor}\left(\Psi(Y_j)-\Es{\Psi(Y_j)}\right)
\Rightarrow  s^{-\alpha}\frac{1}{r!}J_r(\Psi)Z_{r, H}(t).
\end{align*}

In the following, we will see that the first and the second summand in \eqref{eq:Rn-summands} are negligible.
For this, it suffices to show that
\begin{align*}
\lim\limits_{n\rightarrow\infty}\frac{\lfloor nt\rfloor}{d_{n, r}}\Es{\left|\Psi_n(Y_1, s)-s^{-\alpha}\Psi(Y_1)\right|}= 0.
\end{align*}
Note that
\begin{align*}
\Es{\left|\Psi_n(Y_1, s)-s^{-\alpha}\Psi(Y_1)\right|}
&= \int_{\mathbb{R}} \left|\frac{\overline{F}_{\varepsilon}\left(\frac{u_ns}{\sigma(y)}\right)}{\overline{F}_{\varepsilon}(u_n)}-\left(\frac{s}{\sigma(y)}\right)^{-\alpha}\right|\varphi(y)dy.
\end{align*}
Due to  second order regular variation of $\overline{F}_{\varepsilon}$, Lemma \ref{lem:control_ratio_Fz} implies that for any $\varepsilon >0$
\begin{align*}
\left|\frac{\overline{F}_{\varepsilon}\left(\frac{u_ns}{\sigma(y)}\right)}{\overline{F}_{\varepsilon}(u_n)}-\left(\frac{s}{\sigma(y)}\right)^{-\alpha}\right| \leq C\eta^*(u_n)\left(\frac{s}{\sigma(y)}\right)^{-\rho-\alpha}\left(\max\ens{\frac{s}{\sigma(y)},  \frac{\sigma(y)}{s}}\right)^{\varepsilon}.
\end{align*}
Thus, it follows that
\begin{align*}
\sup\limits_{s\geq s_0}\Es{\left|\Psi_n(Y_1, s)-s^{-\alpha}\Psi(Y_1)\right|}
&\leq  C\eta^{*}(u_n)s_0^{-\alpha-\rho -\varepsilon}\int_{\mathbb{R}} \sigma^{\alpha+\rho+\varepsilon}(y)\varphi(y)dy.
\end{align*}
By   \ref{assump:eta_of_un} and  \ref{assump:moments:isigmaY}, i.e., since $\eta^{*}(u_n)=o\pr{d_{n, r}/n}$ and $\Es{\sigma^{\alpha +\rho+\varepsilon}(Y_1)}<\infty$, it holds that
\begin{align*}
\sup\limits_{s\geq s_0}\Es{\left|\Psi_n(Y_1, s)-s^{-\alpha}\Psi(Y_1)\right|}=o\left(\frac{d_{n, r}}{n}\right).
\end{align*}
This completes the proof of Proposition \ref{lem:weak_conv_of_LRD_part}.
\end{proof}

\subsubsection{The martingale part}
The goal of this section is to prove the following proposition:

\begin{proposition}\label{prop:convergence_of_the_martingale_part}
Under the assumptions of Theorem \ref{thm:TEP}, for any $R>1$, the sequence $\sqrt{n\overline{F}_{\varepsilon}\pr{u_n}} M_n\pr{s,t}$, $n\geq 1$,
converges in distribution to $\sqrt{\Es{ \sigma^{\alpha}(Y_1)}}B({s^{-\alpha},t})$
in $D\pr{[1,R]\times [0,1]}$, where $B$ denotes a standard Brownian sheet.
\end{proposition}
First, we establish convergence of the finite dimensional distributions. Then, we proceed with a proof of tightness. For the latter, we use
chaining arguments; see Section~\ref{subsec:tightness_chaining}).
\subsubsection*{The martingale part: Convergence of the finite dimensional distributions}\label{sec:mtg-fidi}

We have to prove that for all positive integers $d_1$ and 
$d_2$, all $1\leq s_{d_1}<\dots<s_{1}$ with $s_i\in \R$ and all 
$0\leq t_1<\dots <t_{d_2}\leq 1$, the vector with entries
$\sqrt{n\overline{F}_{\varepsilon}\pr{u_n}} M_n\pr{s_i,t_j}$, $1\leq i\leq d_1$,  $1\leq j\leq d_2$, converges 
in distribution to $\sqrt{\Es{ \sigma^{\alpha}(Y_1)}}B({s_i^{-\alpha},t_j})$,  $1\leq i\leq d_1$,  $1\leq j\leq d_2$. 
For this, it suffices to consider the case $d_1=d_2$. Indeed, if 
$d_1<d_2$, we include $s_i$, $d_1+1\leq i\leq d_2$, in a decreasing 
order between $s_{d_1}$ and $s_{d_1-1}$, i.e., such that 
$1\leq s_{d_1}< s_{d_2}<\dots<s_{d_1+1}<s_{d_1-1}<\dots <s_{1}$, and if
$d_2<d_1$, we include $t_j$, $d_2+1\leq j\leq d_1$ 
between $t_{d_2-1}$ and $t_{d_2}$ in an increasing order, i.e., 
$t_{d_2-1}<t_{d_2+1}<\dots<t_{d_1}<t_{d_2}\leq 1$. 
Letting $d=\max\ens{d_1,d_2}$, the convergence in distribution of 
the vector with entries $\sqrt{n\overline{F}_{\varepsilon}\pr{u_n}} M_n\pr{s_i,t_j}$,  $1\leq i\leq d$,  $1\leq j\leq d$,
 to $\sqrt{\Es{ \sigma^{\alpha}(Y_1)}}B({s_i^{-\alpha},t_j})$,  $1\leq i\leq d$,  $1\leq j\leq d$, implies  convergence of 
$\sqrt{n\overline{F}_{\varepsilon}\pr{u_n}} M_n\pr{s_i,t_j}$,  $1\leq i\leq d_1$,  $1\leq j\leq d_2$,
 to $\sqrt{\Es{ \sigma^{\alpha}(Y_1)}}B({s_i^{-\alpha},t_j})$,  $1\leq i\leq d_1$,  $1\leq j\leq d_2$.

Let $d\geq 1$ be an integer
and let $s_1, \ldots, s_d$ and $t_1, \ldots, t_d$ be real numbers such that
$1\leq s_d <\dots<s_1\leq R$ and $0\eqdef t_0\leq t_1<\dots<t_d\leq 1$.
Define t intervals $I_{n,1}\defeq\pr{s_1u_n, \infty}$, and for
$2\leq i\leq d$, let $I_{n,i}\defeq\left(s_iu_n,s_{i-1}u_n\right]$. Moreover,  define
 random variables
\begin{equation*}
Z_{i,j}^n\defeq\frac 1{\sqrt{n\overline{F}_{\varepsilon}\pr{u_n}}}\sum_{k=\pe{nt_{j-1}}+1}^{\pe{nt_{j}}}
\pr{
\mathbf{1}\ens{X_k\in I_{n,i}}-\Es{\mathbf{1}\ens{X_k\in I_{n,i}}\mid\mathcal F_{k-1}}
}
\end{equation*}
for $1\leq i,j\leq d$.

\begin{lemma}\label{lem:finite_dim_distrib_disjoint}
The sequence of random vectors $\pr{Z_{i,j}^n}_{1\leq i,j\leq d}$, $n\geq 1$, converges in
distribution to $\pr{Z_{i,j}}_{1\leq i,j\leq d}$, where the  random variables
$Z_{i,j}$, $1\leq i,j\leq d$, are independent, and for all $i,j\in\ens{1,\dots,d}$,
the random variable $Z_{i,j}$ is Gaussian, centered, and has variance $\Es{\sigma^\alpha\pr{Y_1}}\pr{t_j-t_{j-1}}
\pr{s_i^{-\alpha}-s_{i-1}^{-\alpha}}$, where $s_{0}^{-\alpha}=0$.
\end{lemma}

This lemma directly provides convergence of the finite dimensional distributions, after
having applied the continuous mapping theorem to the function
\begin{align*}
\pr{x_{i,j}}_{i,j=1}^d\mapsto \pr{\sum_{i'=1}^i\sum_{j'=1}^jx_{i',j'}   }_{i,j=1}^d.
\end{align*}

 \begin{proof}[Proof of Lemma~\ref{lem:finite_dim_distrib_disjoint}]
 By the Cram\'{e}r-Wold device, we have to prove that for each collection of real numbers
 $a_{i,j}$, $1\leq i,j\leq d$, the sequence $\sum_{i,j=1}^da_{i,j}Z_{i,j}^n$, $n\geq 1$,
 converges to a normal distribution with mean zero and variance
 \begin{equation}\label{eq:limiting-variance}
 \sigma^2\defeq\Es{\sigma^\alpha\pr{Y_1}}\sum_{i,j=1}^d \pr{t_j-t_{j-1}}
\pr{s_i^{-\alpha}-s_{i-1}^{-\alpha}}a_{i,j}^2.
 \end{equation}

We will prove this by applying the following central limit theorem for martingale difference
 arrays. For this, recall that $\Delta_{n,k}$, $n\geq 1$, $1\leq k\leq n$, is a martingale difference
array with respect to the filtration $\mathcal F_{n,k}$, $n\geq 1$,  $0\leq k\leq n$, if
for all $n$, $\mathcal F_{n,k}\subset \mathcal F_{n,k+1}$, $1\leq k\leq n-1$,
$\Delta_{n,k}$ is $\mathcal F_{n,k}$-measurable and $\Es{\Delta_{n,k}
\mid \mathcal F_{n,k-1}}=0$.
\begin{theorem}[Theorem VIII. 1 in \cite{pollard:1984}]
Let $\pr{\Delta_{n,k}}_{n\geq 1, 1\leq k\leq n}$ be a martingale difference
array with respect to the filtration $\pr{\mathcal F_{n,k}}_{n\geq 1,0\leq k\leq n}$,
such that $\Delta_{n,k}$ is square integrable for all $n$ and $k$. Moreover, assume that
\begin{enumerate}
\item
for each positive $\epsilon$,
\begin{equation}\label{itm_condition_1_TLC_martingale_arrays}
\sum_{k=1}^n\Es{\Delta_{n,k}^2
\mathbf{1}\ens{\abs{\Delta_{n,k}}>\epsilon}
\mid \mathcal F_{n,k-1}}\to 0\mbox{ in probability};
\end{equation}
\item
\begin{equation}\label{itm_condition_2_TLC_martingale_arrays}
\sum_{k=1}^n\Es{\Delta_{n,k}^2
\mid \mathcal F_{n,k-1}}\to\sigma^2\mbox{ in probability}.
\end{equation}
\end{enumerate}
Then, the sequence $\sum_{k=1}^n\Delta_{n,k}$, $n\geq 1$, converges
in distribution to a normal law with mean zero and variance $\sigma^2$ defined in \eqref{eq:limiting-variance}.
\end{theorem}

 We express $\widetilde{Z}_n\defeq\sum_{i,j=1}^da_{i,j}Z_{i,j}^n$ as a sum of martingale differences. For $1\leq i\leq d$, define
  \begin{equation*}
 D_{n,i,k}\defeq \frac 1{\sqrt{n\overline{F}_{\varepsilon}\pr{u_n}}}
 \left(\mathbf 1\ens{X_k\in I_{n,i}} -\Es{
 \mathbf 1\ens{ X_k\in I_{n,i}}\mid \mathcal F_{k-1}}  \right).
 \end{equation*}
If for some $j\in \{1, \ldots, d\}$ the integer $k$ satisfies
$\pe{nt_j}+1\leq k\leq \pe{nt_{j+1}}$, then we define
 \begin{equation*}
\Delta_{n,k}\defeq\sum_{i=1}^da_{i,j}D_{n,i,k},
\end{equation*}
and if $\pe{nt_{d}}+1\leq k\leq n$, we define $\Delta_{n,k}\defeq 0$. 

In this way,
$\widetilde{Z}_n=\sum_{k=1}^n\Delta_{n,k}$ and defining $\mathcal F_{n,j}$ as
the $\sigma$-algebra $\mathcal F_j$ given by \eqref{eq:definition_of_Fj},
 the array $\Delta_{n,k}$, $n\geq 1$, $1\leq k\leq n$, is a square integrable
 martingale difference array
 with respect to the filtration $\mathcal F_{n,k}$, $n\geq 1$, $0\leq k\leq n$.
 Let us check \eqref{itm_condition_1_TLC_martingale_arrays}. Observe that for all
 $1\leq k\leq n$ and $1\leq i\leq d$, $\abs{D_{n,i,k}}\leq 2
 \pr{n\overline{F}_{\varepsilon}\pr{u_n}}^{-1/2}$. Hence, for all $1\leq k\leq n$,
 \begin{equation*}
 \abs{\Delta_{n,k}}\leq 2\sum_{i,j=1}^d\abs{a_{i,j}}
 \pr{n\overline{F}_{\varepsilon}\pr{u_n}}^{-1/2}.
 \end{equation*}
 Consequently, for a fixed $\epsilon$, when $n$ is such that $
 2\sum_{i,j=1}^d\abs{a_{i,j}}
 \pr{n\overline{F}_{\varepsilon}\pr{u_n}}^{-1/2}<\epsilon$, the indicator
 $\mathbf{1}\ens{\abs{\Delta_{n,k}}>\epsilon}$ vanishes and hence
 \eqref{itm_condition_1_TLC_martingale_arrays} holds.

 Let us check \eqref{itm_condition_2_TLC_martingale_arrays}. It suffices to prove that
 for all $j\in \ens{1,\dots,d}$
 \begin{equation}\label{eq:sum-of-squares}
 \Es{\abs{
\sum_{k=\pe{nt_{j-1}}+1   } ^{\pe{nt_{j}}}
 \Es{\Delta_{n.k}^2\mid \mathcal F_{k-1}} -\Es{\sigma^\alpha\pr{Y_1}}\pr{t_j-t_{j-1}}\sum_{i=1}^d
 a_{i,j}^2\pr{s_i^{-\alpha}-s_{i-1}^{-\alpha}}
 }}\to 0.
 \end{equation}
 To this aim, we decompose $\Es{\Delta_{n.k}^2\mid \mathcal F_{k-1}}$:
 \begin{multline*}
 \Es{\Delta_{n.k}^2\mid \mathcal F_{k-1}}=
 \sum_{i=1}^d a_{i,j}^2\Es{D_{n,i,k}^2\mid\mathcal F_{k-1}}
 +2\sum_{1\leq i_1<i_2\leq d}a_{i_1,j}a_{i_2,j}
 \Es{D_{n,i_1,k}D_{n,i_2,k}\mid\mathcal F_{k-1}}.
 \end{multline*}
 Observe that using \eqref{eq:conditional_expectation_martingale_part}
 \begin{multline*}
 \Es{D_{n,i,k}^2\mid\mathcal F_{k-1}}=\frac 1{n\overline{F}_{\varepsilon}\pr{u_n}}
  \pr{\overline{F}_{\varepsilon}\pr{\frac{u_ns_i}{\sigma\pr{Y_k}}     }
  -\overline{F}_{\varepsilon}\pr{\frac{u_ns_{i-1}}{\sigma\pr{Y_k}}     }
   }\\
   -\frac 1{n\overline{F}_{\varepsilon}\pr{u_n}}
  \pr{\overline{F}_{\varepsilon}\pr{\frac{u_ns_i}{\sigma\pr{Y_k}}     }
  -\overline{F}_{\varepsilon}\pr{\frac{u_ns_{i-1}}{\sigma\pr{Y_k}}     }
   }^2,
 \end{multline*}
where we set $\overline{F}_{\varepsilon}\pr{s_0x}=0$ for all $x$. Moreover, it holds that
 \begin{multline*}
  \Es{D_{n,i_1,k}D_{n,i_2,k}\mid\mathcal F_{k-1}}\\
  =\frac 1{n\overline{F}_{\varepsilon}\pr{u_n}}
  \pr{\overline{F}_{\varepsilon}\pr{\frac{u_ns_{i_1}}{\sigma\pr{Y_k}}     }
  -\overline{F}_{\varepsilon}\pr{\frac{u_ns_{i_1-1}}{\sigma\pr{Y_k}}     }
   }\pr{\overline{F}_{\varepsilon}\pr{\frac{u_ns_{i_2}}{\sigma\pr{Y_k}}}
  -\overline{F}_{\varepsilon}\pr{\frac{u_ns_{i_2-1}}{\sigma\pr{Y_k}}     }
   }.
 \end{multline*}
Hence, the expression on the left hand side of \eqref{eq:sum-of-squares} is bounded by
 \begin{equation*}
 \Es{\abs{R_{n,1}}}+\Es{\abs{R_{n,2}}}+\Es{\abs{R_{n,3}}},
 \end{equation*}
 where
 \begin{align*}
 R_{n,1}\defeq \sum_{i=1}^d
 a_{i,j}^2\frac 1{n\overline{F}_{\varepsilon}\pr{u_n}}
\sum_{k=\pe{nt_{j-1}}+1   } ^{\pe{nt_{j}}}   \pr{\overline{F}_{\varepsilon}\pr{\frac{u_ns_i}{\sigma\pr{Y_k}}     }
  -\overline{F}_{\varepsilon}\pr{\frac{u_ns_{i-1}}{\sigma\pr{Y_k}}     }
   }
\\-\Es{\sigma^\alpha\pr{Y_1}}\pr{t_j-t_{j-1}}\sum_{i=1}^d
 a_{i,j}^2\pr{s_i^{-\alpha}-s_{i-1}^{-\alpha}},
 \end{align*}
 \begin{align*}
&R_{n,2}\defeq \sum_{i=1}^da_{i,j}^2 \frac 1{n\overline{F}_{\varepsilon}\pr{u_n}}
\sum_{k=\pe{nt_{j-1}}+1   } ^{\pe{nt_{j}}}   \pr{\overline{F}_{\varepsilon}\pr{\frac{u_ns_i}{\sigma\pr{Y_k}}     }
  -\overline{F}_{\varepsilon}\pr{\frac{u_ns_{i-1}}{\sigma\pr{Y_k}}     }
   }^2, \\
&R_{n,3}\defeq \frac 1{n\overline{F}_{\varepsilon}\pr{u_n}}\sum_{k=\pe{nt_{j-1}}+1   } ^{\pe{nt_{j}}}\sum_{1\leq i_1<i_2\leq d}a_{i_1,j}a_{i_2,j}\\
&\phantom{=}\times \pr{\overline{F}_{\varepsilon}\pr{\frac{u_ns_{i_1}}{\sigma\pr{Y_k}}     }
  -\overline{F}_{\varepsilon}\pr{\frac{u_ns_{i_1-1}}{\sigma\pr{Y_k}}     }
   }\pr{\overline{F}_{\varepsilon}\pr{\frac{u_ns_{i_2}}{\sigma\pr{Y_k}}}
  -\overline{F}_{\varepsilon}\pr{\frac{u_ns_{i_2-1}}{\sigma\pr{Y_k}}     }
   }\;.
\end{align*}
 Here we omit the dependence on $j\in\ens{1,\dots,d}$ in $R_{n,1}$ and
 $R_{n,2}$ in order to ease notations.
 We start with $R_{n,1}$. Using stationarity, it suffices to prove that for all $i\in\ens{1,
 \dots,d}$, $\Es{\abs{R_{n,1,i}}}\to 0$, where
 \begin{multline*}
R_{n,1,i}\defeq
 \frac 1{n\overline{F}_{\varepsilon}\pr{u_n}}
\sum_{k=1   } ^{ \pe{nt_{j}}-\pe{nt_{j-1}}}   \pr{\overline{F}_{\varepsilon}\pr{\frac{u_ns_i}{\sigma\pr{Y_k}}     }
  -\overline{F}_{\varepsilon}\pr{\frac{u_ns_{i-1}}{\sigma\pr{Y_k}}     }
   }\\
   -\Es{\sigma^\alpha\pr{Y_1}}\pr{t_j-t_{j-1}} \pr{s_i^{-\alpha}-s_{i-1}^{-\alpha}}.
 \end{multline*}
 Applying Lemma~\ref{lem:control_ratio_difference_Fz} with $t=u_n$,
 $a=s_i/\sigma\pr{Y_k}$,  $b=s_{i-1}/\sigma\pr{Y_k}$ and $\epsilon=1$, we get
 \begin{multline*}
 \abs{R_{n,1,i}}\leq \abs{
\frac 1n\sum_{k=1   } ^{\pe{nt_{j}}-\pe{nt_{j-1}}}\pr{
\frac{\sigma^\alpha\pr{Y_k}}{s_{i}^\alpha}-
\frac{\sigma^\alpha\pr{Y_k}}{s_{i-1}^\alpha}
} -\pr{t_j-t_{j-1}} \pr{s_i^{-\alpha}-s_{i-1}^{-\alpha}}\Es{\sigma^\alpha\pr{Y_1}}}\\
+C\eta^*\pr{u_n}
\frac 1n\sum_{k=1   } ^{\pe{nt_{j}}-\pe{nt_{j-1}}}
\pr{\max\ens{\frac{\sigma\pr{Y_k}}{s_i}, 1}  }^{  \alpha+\rho+1}\frac{s_i-s_{i-1}}{\sigma\pr{Y_k}}.
 \end{multline*}
By \ref{assump:moments:isigmaY} and \ref{assump:moments:inverse_sigmaY_2+delta} we have
 \begin{align*}
 \Es{\frac{\left(\max\ens{\sigma\pr{Y_0}, 1}\right)^{ \alpha+ \rho+1}}{\sigma\pr{Y_0}}}<\infty.
 \end{align*}

Using the ergodic theorem for the first term, the fact that $\eta^*(u_n)\to 0$ allows us to conclude that
$\Es{\abs{R_{n,1}}}\to 0$ as $n$ goes to infinity.

The treatment of $R_{n,2}$ and $R_{n,3}$ is the same: we take expectations
and bound one of the factors $ \overline{F}_{\varepsilon}\pr{\frac{u_ns_i}{\sigma\pr{Y_k}}     }
  -\overline{F}_{\varepsilon}\pr{\frac{u_ns_{i-1}}{\sigma\pr{Y_k}}     }
$ using \eqref{eq:control_ratio_Fz_simplified} in  Lemma~\ref{lem:control_ratio_difference_Fz}  when $i=1$ and
\eqref{eq:control_ratio_difference_Fz_simplified} in Lemma~\ref{lem:control_ratio_difference_Fz}  when $2\leq i\leq d$. We then conclude
by the dominated convergence theorem.

This finishes the proof of Lemma~\ref{lem:finite_dim_distrib_disjoint}.
 \end{proof}

 \subsubsection*{The martingale part: Tightness}
 \label{subsec:tightness_chaining}

\begin{lemma}\label{lem:convergence_of_the_martingale_part}
Under the assumptions of Theorem \ref{thm:TEP}, for any $R>1$, the sequence 
\begin{align*}
\sqrt{n\overline{F}_{\varepsilon}\pr{u_n}} M_n\pr{s,t}, \ n\geq 1,
\end{align*}
is tight 
in $D\pr{[1,R]\times [0,1]}$.
\end{lemma}
\begin{proof}
 Define
\begin{align*}
m_n\pr{s,t}\defeq\sqrt{n\overline{F}_{\varepsilon}\pr{u_n}} M_n\pr{s,t}.
\end{align*}
In order to prove tightness of $m_n\pr{s,t}$,  we validate the following tightness criterion: for all  $\epsilon>0$
\begin{equation*}
\lim_{\delta\to 0}\limsup_{n\to \infty}
\PP\pr{\sup_{\substack{\abs{s_2-s_1}<\delta\\ 1\leq s_1,s_2\leq R  } }
  \sup_{\substack{\abs{t_2-t_1}<\delta\\ 0\leq t_1,t_2\leq 1  }}\abs{
m_n\pr{s_2,t_2}-m_n\pr{s_1,t_1}
   }>\epsilon     }=0.
\end{equation*}
Writing
\begin{equation*}
m_n\pr{s_2,t_2}-m_n\pr{s_1,t_1}=m_n\pr{s_2,t_2}-m_n\pr{s_1,t_2}+m_n\pr{s_1,t_2}
-m_n\pr{s_1,t_1},
\end{equation*}
it suffices to show
\begin{equation}\label{eq:tightness_Mn_part1}
\lim_{\delta\to 0}\limsup_{n\to \infty}
\PP\pr{\sup_{\substack{\abs{s_2-s_1}<\delta\\ 1\leq s_1,s_2\leq R  } }
  \sup_{t\in [0,1]}\abs{
m_n\pr{s_2,t}-m_n\pr{s_1,t}
   }>\epsilon     }=0
\end{equation}
and
\begin{equation}\label{eq:tightness_Mn_part2}
\lim_{\delta\to 0}\limsup_{n\to \infty}
\PP\pr{\sup_{ 1\leq s \leq R  }
  \sup_{\substack{\abs{t_2-t_1}<\delta\\ 0\leq t_1,t_2\leq 1  }}\abs{
m_n\pr{s,t_2}-m_n\pr{s,t_1}
   }>\epsilon     }=0.
\end{equation}
\subsubsection{Proof of \eqref{eq:tightness_Mn_part1}.}

In order to prove \eqref{eq:tightness_Mn_part1}, we apply a chaining technique.

For this, we define the  intervals
\begin{align*}
I_{1,k}\defeq[1+2k\delta, 1+2(k+1)\delta] \ \text{ and } \ I_{2, k}\defeq[1+(2k+1)\delta, 1+(2(k+1)+1)\delta]
\end{align*}
for $k=0, \ldots, L_{\delta}\defeq\lfloor\frac{R-1}{2\delta}\rfloor$.
Then, the expression inside $\PP$ in \eqref{eq:tightness_Mn_part1} is bounded by
\begin{align}\label{eq:decomposition-1}
\max\limits_{0\leq k\leq L_{\delta}}\sup_{\substack{s_1, s_2\in I_{1, k}\\  } }
  \sup_{t\in [0,1]}\abs{
m_n\pr{s_2,t}-m_n\pr{s_1,t}
   }
   +\max\limits_{0\leq k\leq L_{\delta}}\sup_{\substack{s_1, s_2\in I_{2, k}\\  } }
  \sup_{t\in [0,1]}\abs{
m_n\pr{s_2,t}-m_n\pr{s_1,t}
   }.
\end{align}
In the following, we consider the first summand only, since for the second summand  analogous considerations hold, i.e., 
it remains to show that
\begin{equation*}
\lim_{\delta\to 0}\limsup_{n\to \infty}
\PP\pr{\max\limits_{0\leq k\leq L_{\delta}}\sup_{\substack{s_1, s_2\in I_{1, k}\\  } }
  \sup_{t\in [0,1]}\abs{
m_n\pr{s_2,t}-m_n\pr{s_1,t}
   }>\epsilon     }=0.
\end{equation*}
For this, it suffices to show that
\begin{align*}
\lim_{\delta\to 0}\limsup_{n\to \infty}\frac{1}{\delta}
\max_{0\leq k\leq L_\delta}
\PP\pr{\sup_{\substack{s_1, s_2\in I_{1, k}\\  } }
  \sup_{t\in [0,1]}\abs{
m_n\pr{s_2,t}-m_n\pr{s_1,t}
   }>\epsilon     }=0.
\end{align*}
We write $I_{1, k}=[a_k, a_{k+1}]$, i.e., $a_k\defeq 1+2k\delta$ and $a_{k+1}\defeq 1+2(k+1)\delta$.
Note that
\begin{align*}
\sup_{\substack{s_1, s_2\in I_{1, k}\\  } }
  \sup_{t\in [0,1]}\abs{
m_n\pr{s_2,t}-m_n\pr{s_1,t}
   }\leq 2\sup\limits_{x\in [0, 2\delta]}  \sup_{t\in [0,1]}\abs{
m_n\pr{a_k,t}-m_n\pr{a_k+x,t}
   }.
\end{align*}
Define refining partitions $x_i(k)$ for $k=0, \ldots, K_n$ with $K_n\rightarrow\infty$, for $n\rightarrow\infty$, by
\begin{align*}
x_i(k)\defeq\sup\left\{x\in [1, R]\left|\right. \frac{\overline{F}_{\varepsilon}(u_nx)}{\overline{F}_{\varepsilon}(u_n)}\geq\frac{i}{2^k}\delta\right\}, \ i=0, \ldots, \lfloor 2^k \delta^{-1}\rfloor,
\end{align*}
and choose  $i_k(x)$ such that
\begin{align*}
a_k+x\in \left(x_{i_k(x)+1}(k), x_{i_k(x)}(k)\right].
\end{align*}

From the definition of $m_n$ we obtain
\begin{align*}
\left| m_n(y, t)-m_n(x, t)\right|
=\left|\frac{1}{\sqrt{n\overline{F}_{\varepsilon}(u_n)}}\sum\limits_{j=1}^{\lfloor nt\rfloor}\left(\1\ens{u_nx<X_j\leq u_ny}-
\Es{\1\ens{u_nx<X_j\leq u_n y }| \mathcal{F}_{j-1} }\right)\right|.
\end{align*}
Then, with $m_n(x)\defeq\sup_{t\in [0, 1]}\left|m_n(x, t)\right|$ and $\overline{m}_n(x, y)\defeq\sup_{t\in [0, 1]}\left|m_n(y, t)-m_n(x, t)\right|$,
it follows that
\begin{align*}
&\sup\limits_{t\in [0,1 ]}\abs{m_n(a_k, t)-m_n(a_k+x, t)}\\
&\leq \overline{m}_n(a_k, x_{i_{K_n}(x)+1}(K_n))+\sum\limits_{k=1}^{K_n} \overline{m}_n(x_{i_{k}(x)+1}(k), x_{i_{k-1}(x)+1}(k-1)) + \overline{m}_n(x_{i_0(x)+1}(0), a_k+x).
\end{align*}
As a result, we have
\begin{align}\label{eq:border_terms}
&\PP\pr{\sup\limits_{x\in [0, 2\delta]}\sup\limits_{t\in [0,1 ]}\abs{m_n(a_k, t)-m_n(a_k+x, t)}>\epsilon}\notag\leq
\PP\left( \sup\limits_{x\in [0, 2\delta]}  \overline{m}_n(a_k, x_{i_{K_n}(x)+1}(K_n))>\frac{\epsilon}{4}\right)\notag\\
&+
\sum\limits_{k=1}^{K_n} \PP\left(\sup\limits_{x\in [0, 2\delta]}\overline{m}_n(x_{i_{k}(x)+1}(k), x_{i_{k-1}(x)+1}(k-1)) >\frac{\epsilon}{(k+3)^2}\right)\notag\\
&+ \PP\left(\sup\limits_{x\in [0, 2\delta]}\overline{m}_n(x_{i_0(x)+1}(0), a_k+x)>\epsilon-\sum\limits_{k=0}^{\infty}\frac{\epsilon}{(k+3)^2}-\frac{\epsilon}{4}\right).
\end{align}
Since
\begin{align*}
\sum\limits_{k=0}^{\infty}\frac{\epsilon}{(k+3)^2}\leq \frac{\epsilon}{2},
\end{align*}
it follows that
\begin{align*}
\PP\left(\sup\limits_{x\in [0, 2\delta]}\overline{m}_n(x_{i_0(x)+1}(0), a_k+x)>\epsilon-\sum\limits_{k=0}^{\infty}\frac{\epsilon}{(k+3)^2}-\frac{\epsilon}{4}\right)
\leq \PP\left(\sup\limits_{x\in [0, 2\delta]}\overline{m}_n(x_{i_0(x)+1}(0), a_k+x)>\frac{\epsilon}{4}\right).
\end{align*}
Additionally, we conclude that
\begin{align*}
&\sum\limits_{k=1}^{K_n} \PP\left(\sup\limits_{x\in [0, 2\delta]}\overline{m}_n(x_{i_{k}(x)+1}(k), x_{i_{k-1}(x)+1}(k-1)) >\frac{\epsilon}{(k+3)^2}\right)\\
&\leq \sum\limits_{k=1}^{K_n}\sum\limits_{i=0}^{\lfloor 2^k \delta^{-1}\rfloor} \PP\left(\overline{m}_n(x_{i+1}(k), x_{i}(k)) >\frac{\epsilon}{(k+3)^2}
\right).
\end{align*}

For further estimation, we 
 use Freedman's inequality: if
$\pr{d_j,\mathcal F_j}$,  $j\geq 1$, is a martingale difference sequence such that
$\sup_j \abs{d_j}\leq c$, where $c$ is a positive constant, then for all $x,y>0$
\begin{multline*}
\PP\pr{\ens{\max_{1\leq \ell\leq n}\abs{\sum_{j=1}^\ell d_j}>x   }\cap
\ens{\sum_{j=1}^n \Es{d_j^2\mid \mathcal F_{j-1 }}\leq y     }}\leq
2\exp\pr{-\frac{x^2}{2\pr{y+\frac 2{3}cx}  }}\\
\leq 2 \max\ens{\exp\pr{-\frac{x^2}{4y}}, \exp\pr{-\frac 38 x/c}  };
\end{multline*}
see Theorem~1.6 in \cite{freedman:1975}.

For this purpose, we define
\begin{align*}
d_{j, i, k, n}\defeq\1 \ens{u_nx_{i+1}(k)< X_j\leq u_nx_{i}(k) }-\Es{\1 \ens{u_nx_{i+1}(k) <X_j\leq u_nx_{i}(k)}\left|\right. \mathcal{F}_{j-1}}.
\end{align*}
Then, it follows that
\begin{equation*}
\overline{m}_n(x_{i+1}(k), x_{i}(k))=
\max\limits_{1\leq \ell\leq n}\left|\frac{1}{\sqrt{n\overline{F}_{\varepsilon}(u_n)}}\sum\limits_{j=1}^{\ell}d_{j, i, k, n}\right|.
\end{equation*}

Furthermore, for $B>0$, which is chosen later,  define
\begin{align*}
y_{n, k, i}\defeq n\overline{F}_{\varepsilon}(u_n)B\left(\frac{\overline{F}_{\varepsilon}\left(u_nx_{i+1}(k)\right)}{\overline{F}_{\varepsilon}(u_n)}-\frac{\overline{F}_{\varepsilon}\left(u_nx_i(k)\right)}{\overline{F}_{\varepsilon}(u_n)}\right)
\end{align*}
and
\begin{align*}
z_{n, k}\defeq\frac{\epsilon\sqrt{n\overline{F}_{\varepsilon}(u_n)}}{(k+3)^2}.
\end{align*}
Since $\overline{F}_{\varepsilon}$  is  continuous and since we can assume without loss of generality that $\overline{F}_{\varepsilon}$ is ultimately strictly decreasing,
it holds that
\begin{align*}
\frac{\overline{F}_{\varepsilon}\left(u_nx_{i+1}(k)\right)}{\overline{F}_{\varepsilon}(u_n)}-\frac{\overline{F}_{\varepsilon}\left(u_nx_i(k)\right)}{\overline{F}_{\varepsilon}(u_n)}
= \frac{\delta}{2^k}, 
\end{align*}
and, consequently, 
\begin{align*}
y_{n, k, i}=n\overline{F}_{\varepsilon}(u_n)B\frac{\delta}{2^k}.
\end{align*}

For an estimation by Freedman's inequality, we have to specify $K_n$.
Choosing $K_n\defeq\lfloor\log_2\left(\delta
a_n C\right)\rfloor$ for some  constant $C$, where
$a_n$, $n\geq 1$, is a sequence with
$a_n\rightarrow \infty$ and $a_n=o\left(\frac{n}{d_{n, r}}+\sqrt{n\overline{F}_{\varepsilon}(u_n)}\right)$,
it follows (with $c=2$) that
\begin{align*}
\frac{z_{n, k}^2}{4y_{n, k, i}}=\frac{\epsilon^22^k}{4B(k+3)^4\delta}\leq \frac{\epsilon^2C\sqrt{n\overline{F}_{\varepsilon}(u_n)}}{4B(k+3)^4}
\leq\frac{3}{8}\frac{\epsilon\sqrt{n\overline{F}_{\varepsilon}(u_n)}}{2(k+3)^2}=\frac{3}{8}\frac{z_{n, k}}{c}
\end{align*}
for a corresponding choice of $C$.
Therefore, we have
\begin{align*}
\max\ens{\exp\pr{-\frac{z_{n, k}^2}{4y_{n, k, i}}}, \exp\pr{-\frac 38 z_{n, k}/c}  }=
\exp\pr{-\frac{z_{n, k}^2}{4y_{n, k, i}}}.
\end{align*}

As a result, an application of Freedman's inequality yields
\begin{align*}
&\sum\limits_{k=1}^{K_n}\sum\limits_{i=0}^{\lfloor 2^k \delta^{-1}\rfloor} \PP\left(\overline{m}_n(x_{i+1}(k), x_{i}(k)) >\frac{\epsilon}{(k+3)^2}, \sum\limits_{j=1}^n
\Es{d_{j, i, k, n}^{2}\left|\right.\mathcal{F}_{j-1}}\leq y_{n, k, i}\right)\\
&\leq\sum\limits_{k=1}^{K_n}\sum\limits_{i=0}^{\lfloor 2^k \delta^{-1}\rfloor} \PP\left(\max\limits_{1\leq \ell\leq n}\left|\sum\limits_{j=1}^{\ell}d_{j, i, k, n}\right|>\frac{\epsilon \sqrt{n\overline{F}_{\varepsilon}(u_n)}}{(k+3)^2}, \sum\limits_{j=1}^n\Es{d_{j, i, k, n}^{2}\left|\right.\mathcal{F}_{j-1}}\leq y_{n, k, i}\right)\\
&\leq  \sum\limits_{k=1}^{K_n}\left(\lfloor 2^k \delta^{-1}\rfloor +1\right)
\exp\left(-\frac{\epsilon^2}{4(k+3)^4}\frac{2^k}{B\delta}\right).
\end{align*}
Noting that there exists a constant $D>0$  such that for all $k$
\begin{align*}
\left(\lfloor 2^k \delta^{-1}\rfloor +1\right)
\exp\left(-\frac{\epsilon^2}{4(k+3)^4\delta}\frac{2^k}{B}\right)
\leq  \lfloor 2^\frac{k}{2} \delta^{-1}\rfloor
\exp\left(-D\frac{\epsilon^2}{4}\frac{ 2^{\frac{k}{2}}}{B\delta}\right),
\end{align*}
elementary calculations yield
\begin{align*}
&\sum\limits_{k=1}^{K_n}\left(\lfloor 2^k \delta^{-1}\rfloor +1\right)
\exp\left(-\frac{\epsilon^2}{4(k+3)^4}\frac{2^k}{B\delta}\right)\\
&\leq\frac{4B}{\epsilon^2\log(2)}\left(\exp\left(-D\frac{\epsilon^2}{4B\delta}\right)-\exp\left(-D\frac{\epsilon^2}{4B\delta}2^{\frac{K_n}{2}}\right)\right).
\end{align*}
It follows that
\begin{align*}
\lim\limits_{\delta\rightarrow 0}\lim\limits_{n\rightarrow \infty}\frac{1}{\delta}\sum\limits_{k=1}^{K_n}\sum\limits_{i=0}^{\lfloor 2^k \delta^{-1}\rfloor}
\PP\left(\max\limits_{1\leq \ell\leq n}\left|\sum\limits_{j=1}^{\ell}d_{j, i, k, n}\right|>z_{n, k}, \sum\limits_{j=1}^n
\Es{d_{j, i, k, n}^{2}\left|\right.\mathcal{F}_{j-1}}\leq y_{n, k, i}\right)=0.
\end{align*}
 Therefore, in order to finish the proof of \eqref{eq:tightness_Mn_part1},
it remains to show that
\begin{align*}
&\lim\limits_{\delta\rightarrow 0}\lim\limits_{n\rightarrow \infty}\frac{1}{\delta}\sum\limits_{k=1}^{K_n}\sum\limits_{i=0}^{\lfloor 2^k \delta^{-1}\rfloor}
\PP\left( \sum\limits_{j=1}^n\Es{d_{j, i, k, n}^{2}\left|\right.\mathcal{F}_{j-1}}> y_{n, k, i}\right)=0.
\end{align*}

Since for an event $A$ and a $\sigma$-algebra $\mathcal F$
\begin{equation*}
\Es{  \pr{\1\pr{A}-\Es{\1\pr{A}\mid \mathcal F}   }^2\mid \mathcal F  }
=\Es{\mathbf 1\pr{A}\mid \mathcal F}-\left(\Es{\mathbf 1\pr{A}\mid \mathcal F}\right)^2\leq
\Es{\mathbf 1\pr{A} \mid \mathcal F},
\end{equation*}
it holds that
\begin{equation*}
\Es{d_{j, i, k, n}^{2}\left|\right.\mathcal{F}_{j-1}}\leq
\overline{F}_{\varepsilon}\pr{\frac{u_nx_{i+1}(k)}{\sigma\pr{Y_j}}}-
\overline{F}_{\varepsilon}\pr{\frac{u_nx_{i}(k)}{\sigma\pr{Y_j}}}.
\end{equation*}

Therefore, we arrive at
\begin{align*}
\PP\left(\sum\limits_{j=1}^n\Es{d_{j, i, k, n}^{2}\left|\right.\mathcal{F}_{j-1}}> y_{n, k, i}\right)\leq
\PP\left( \frac 1{y_{n, k, i}}
\sum_{j=1}^n\left(\overline{F}_{\varepsilon}\pr{\frac{u_nx_{i+1}(k)}{\sigma\pr{Y_j}}}-
\overline{F}_{\varepsilon}\pr{\frac{u_nx_{i}(k)}{\sigma\pr{Y_j}}}\right)>1  \right).
\end{align*}
Note that
\begin{align}\label{eq:dec_var_cond_chaining}
 \frac 1{y_{n, k, i}}
\sum_{j=1}^n\left(\overline{F}_{\varepsilon}\pr{\frac{u_nx_{i+1}(k)}{\sigma\pr{Y_j}}}-
\overline{F}_{\varepsilon}\pr{\frac{u_nx_{i}(k)}{\sigma\pr{Y_j}}}\right)
=A_{1}(n, k, i)+A_{2}(n, k, i)+A_{3}(n, k, i),
\end{align}
where
\begin{align*}
&A_{1}(n, k, i)\defeq\frac{\overline{F}_{\varepsilon}(u_n)}{y_{n, k, i}}
\sum_{j=1}^n\left\{\frac{\overline{F}_{\varepsilon}\pr{\frac{u_n x_{i+1}(k)}{\sigma\pr{Y_j}}}}{\overline{F}_{\varepsilon}(u_n)}-
\frac{\overline{F}_{\varepsilon}\pr{\frac{u_n x_{i}(k)}{\sigma\pr{Y_j}}}}{\overline{F}_{\varepsilon}(u_n)}-\sigma^{\alpha}(Y_j)\left(x_{i+1}(k)^{-\alpha}-x_{i}(k)^{-\alpha}\right)\right\},\\
&A_{2}(n, k, i)\defeq\frac{\overline{F}_{\varepsilon}(u_n)}{y_{n, k, i}}\sum_{j=1}^n\sigma^{\alpha}(Y_j)\left\{\left(x_{i+1}(k)^{-\alpha}-x_{i}(k)^{-\alpha}\right)-\left(\frac{\overline{F}_{\varepsilon}\left(u_nx_{i+1}(k)\right)}{\overline{F}_{\varepsilon}(u_n)}-\frac{\overline{F}_{\varepsilon}\left(u_nx_{i}(k)\right)}{\overline{F}_{\varepsilon}(u_n)}\right)\right\},\\
&A_{3}(n, k, i)\defeq\frac{\overline{F}_{\varepsilon}(u_n)}{y_{n, k, i}}\sum\limits_{j=1}^n\sigma^{\alpha}(Y_j)\left(\frac{\overline{F}_{\varepsilon}\left(u_nx_{i+1}(k)\right)}{\overline{F}_{\varepsilon}(u_n)}-\frac{\overline{F}_{\varepsilon}\left(u_nx_{i}(k)\right)}{\overline{F}_{\varepsilon}(u_n)}\right),
\end{align*}
so that
\begin{align}
&\PP\pr{\frac 1{y_{n, k, i}}
\sum_{j=1}^n\left[\overline{F}_{\varepsilon}\pr{\frac{u_n x_{i+1}(k)}{\sigma\pr{Y_j}}}-
\overline{F}_{\varepsilon}\pr{\frac{u_nx_{i}(k)}{\sigma\pr{Y_j}}}\right]>1 }\notag\\
&\leq
\PP\pr{\left|A_{1}(n, k,  i)\right|>\frac{1}{3} }
+\PP\pr{\left|A_{2}(n, k, i)\right|>\frac{1}{3} }
+\PP\pr{\left|A_{3}(n, k, i)\right|>\frac{1}{3} }.\label{eq:summands}
\end{align}
According to Lemma \ref{lem:control_ratio_difference_Fz} in the appendix it holds that
\begin{align*}
\left|A_{2}(n, k,  i)\right|
&\leq \frac{2^{k}\overline{F}_{\varepsilon}(u_n)}{n\overline{F}_{\varepsilon}(u_n)B\delta}\mathcal{O}(\eta^{\star}(u_n))(x_{i+1}(k)-x_{i}(k))
\sum_{j=1}^n\sigma^{\alpha}(Y_j).
\end{align*}

Therefore, given Assumption \ref{assump:eta_of_un}, it follows that
\begin{align*}
&\lim\limits_{n\rightarrow \infty}\sum\limits_{k=1}^{K_n}\sum\limits_{i=0}^{\lfloor 2^k \delta^{-1}\rfloor}\frac{1}{\delta}
\PP\pr{\left|A_{2}(n, k,  i)\right|>\frac{1}{3} }\\
&\leq \lim\limits_{n\rightarrow \infty}\mathcal{O}(\eta^{\star}(u_n))\frac{1}{\delta}\sum\limits_{k=1}^{K_n}\sum\limits_{i=0}^{\lfloor 2^k \delta^{-1}\rfloor}\frac{2^{k}}{B\delta}(x_{i+1}(k)-x_{i}(k))
\Es{\sigma^{\alpha}\pr{Y_1}}\\
&=\lim\limits_{n\rightarrow \infty}\mathcal{O}(\eta^{\star}(u_n))\mathcal{O}\left(\frac{1}{\delta}\sum\limits_{k=1}^{K_n}\frac{2^{k}}{B\delta}\right)
= \lim\limits_{n\rightarrow \infty}\mathcal{O}(\eta^{\star}(u_n))\mathcal{O}\left(\frac{1}{\delta}\frac{2^{K_n}}{B\delta}\right)=0.
\end{align*}

For the first summand in \eqref{eq:dec_var_cond_chaining}, it follows by  Lemma \ref{lem:control_ratio_difference_Fz} in the appendix  that
\begin{align*}
\left|A_{1}(n, k,  i)\right|
&\leq \frac{\overline{F}_{\varepsilon}(u_n)}{y_{n, k, i}}\sum\limits_{j=1}^nC\eta^{\star}(u_n)\left(\frac{x_{i+1}(k)}{\sigma(Y_j)}-\frac{x_{i}(k)}{\sigma(Y_j)}\right)\left(
\min\ens{\frac{x_{i+1}(k)}{\sigma(Y_j)},1}\right)^{-\alpha-\rho-\epsilon}\\
&\leq \frac{\overline{F}_{\varepsilon}(u_n)}{n\overline{F}_{\varepsilon}(u_n)B\frac{\delta}{2^k}}C\eta^{\star}(u_n)(x_{i+1}(k)-x_{i}(k))\sum\limits_{j=1}^n\sigma^{-1}(Y_j)\left(
\max\ens{\sigma\pr{Y_j}, 1}\right)^{\alpha+\rho+\epsilon}\\
&= \frac{2^k}{B\delta}C\eta^{\star}(u_n)(x_{i+1}(k)-x_{i}(k))\frac{1}{n}\sum\limits_{j=1}^n\sigma^{-1}\pr{Y_j}
\left(\max\ens{\sigma\pr{Y_j}, 1}\right)^{\alpha+\rho+\epsilon}.
\end{align*}
According to Assumptions \ref{assump:moments:isigmaY} and \ref{assump:moments:inverse_sigmaY_2+delta} the expectation of the summands on the right-hand side is finite and hence, for each $\delta>0$,
\begin{align*}
&\lim\limits_{n\rightarrow \infty}\sum\limits_{k=1}^{K_n}\sum\limits_{i=0}^{\lfloor 2^k \delta^{-1}\rfloor}\frac{1}{\delta}
\PP\pr{\left|A_{1}(n, k,  i)\right|>\frac{1}{3} }\\
&\leq \lim\limits_{n\rightarrow \infty}\frac{1}{\delta}\sum\limits_{k=1}^{K_n}\sum\limits_{i=0}^{\lfloor 2^k \delta^{-1}\rfloor}\frac{2^{k}}{nB\delta}\mathcal{O}(\eta^{\star}(u_n))(x_{i+1}(k)-x_{i}(k))
\sum_{j=1}^n\Es{\sigma^{-1}(Y_j)\left(\max\ens{\sigma\pr{Y_j}, 1}\right)^{\alpha+\rho+\epsilon}}\\
&= \lim\limits_{n\rightarrow \infty}\mathcal{O}(\eta^{\star}(u_n))\mathcal{O}\left(\frac{1}{\delta}\sum\limits_{k=1}^{K_n}\frac{2^{k}}{B\delta}\right)
=  \lim\limits_{n\rightarrow \infty}\mathcal{O}(\eta^{\star}(u_n))\mathcal{O}\left(\frac{1}{\delta}\frac{2^{K_n}}{B\delta}
\right)\\
&=  \lim\limits_{n\rightarrow \infty}\mathcal{O}(\eta^{\star}(u_n))\mathcal{O}\left(\frac{a_n}{B\delta}
\right)=0.
\end{align*}
Finally, we consider the last
summand in \eqref{eq:dec_var_cond_chaining}:
\begin{equation*}
A_{3}(n, k, i)=\frac{1}{nB}\sum\limits_{j=1}^n\sigma^{\alpha}\pr{Y_j}.
\end{equation*}
Obviously, $A_{3}(n, k, i)$ depends  neither on $k$ nor on $i$. 
Due to the non-central limit theorem in \cite{taqqu:1979}, it holds that
\begin{align*}
\frac{1}{nB}\sum\limits_{j=1}^n\pr{\sigma^{\alpha}(Y_j)-\Es{\sigma^{\alpha}\pr{Y_1}}}=\mathcal{O}_P\left(\frac{d_{n, r}}{n}\right).
\end{align*}
Choosing $B>6\Es{\sigma^{\alpha}(Y_1)}$, it follows that
\begin{align*}
\PP\pr{\frac{1}{nB}\sum\limits_{j=1}^n\sigma^{\alpha}\pr{Y_j}>\frac{1}{3} }\leq\PP\pr{\frac{1}{nB}\sum\limits_{j=1}^n\left(\sigma^{\alpha}\pr{Y_j}-\Es{\sigma^{\alpha}\pr{Y_1}}\right)>\frac{1}{6}}=\mathcal{O}\left(\frac{d_{n, r}}{n}\right).
\end{align*}
Hence, we have
\begin{align*}
&\lim\limits_{n\rightarrow \infty}\sum\limits_{k=1}^{K_n}\sum\limits_{i=0}^{\frac{2^k}{\delta}-1}\frac{1}{\delta}\PP\pr{\left|A_{3}(n, k,  i)\right|>\frac{1}{3} }\\
&=\lim\limits_{n\rightarrow \infty}\PP\pr{\frac{1}{nB}\sum\limits_{j=1}^n\sigma^{\alpha}\pr{Y_j}>\frac{1}{3} }\sum\limits_{k=1}^{K_n}\sum\limits_{i=0}^{\frac{2^k}{\delta}-1}\frac{1}{\delta}\\
&=\lim\limits_{n\rightarrow \infty}\mathcal{O}\left(\frac{d_{n, r}}{n}\right)\sum\limits_{k=1}^{K_n}\frac{2^k}{\delta^2}=\lim\limits_{n\rightarrow \infty}\mathcal{O}\left(\frac{d_{n, r}}{n}\frac{2^{K_n}}{\delta^2}\right)=0.
\end{align*}

All in all, the previous considerations show that the second summand in \eqref{eq:border_terms} converges to $0$. Since the other 
 terms in \eqref{eq:border_terms} can be treated analogously, this finishes the proof of \eqref{eq:tightness_Mn_part1}.

\subsubsection{Proof of \eqref{eq:tightness_Mn_part2}.}
Initially, note that
\begin{align*}
&\sup_{\substack{\abs{t_2-t_1}<\delta\\ 0\leq t_1,t_2\leq 1  }}\abs{
m_n\pr{s,t_2}-m_n\pr{s,t_1}
   }\\
   &=
\sup_{\substack{\abs{t_2-t_1}<\delta\\ 0\leq t_1,t_2\leq 1  }}\abs{    \frac{1}{\sqrt{n\overline{F}_{\varepsilon}\pr{u_n}}}
\sum\limits_{j=\lfloor nt_1\rfloor+1}^{\lfloor nt_2\rfloor}\left(\1\left\{X_j>u_ns\right\} -
\Es{\1\left\{X_j>u_n s\right\} \left|\right. \mathcal{F}_{j-1}}\right)}.
\end{align*}

As before, we apply a chaining technique in order to prove \eqref{eq:tightness_Mn_part2}.
For this, we define the intervals 
\begin{align*}
I_{1,k}\defeq[2k\delta, 2(k+1)\delta] \ \text{ and } \ I_{2, k}\defeq[(2k+1)\delta, (2(k+1)+1)\delta].
\end{align*}
for $k=0, \ldots, L_{\delta}\defeq\lfloor\frac{1}{2\delta}\rfloor$.

Then, similarly to \eqref{eq:decomposition-1}, it holds that
\begin{align*}
&\sup_{ 1\leq s \leq R  } \sup_{\substack{\abs{t_2-t_1}<\delta\\ 0\leq t_1,t_2\leq 1  }}\abs{
m_n\pr{s,t_2}-m_n\pr{s,t_1}
   }\\
&   \leq \sup_{ 1\leq s \leq R  } \max\limits_{
0\leq k\leq L_{\delta}}\sup_{\substack{t_1, t_2\in I_{1, k}  } }
\abs{
m_n\pr{s,t_2}-m_n\pr{s,t_1}
   }+\sup_{ 1\leq s \leq R  } \max\limits_{
   0\leq k\leq L_{\delta}}\sup_{\substack{t_1, t_2\in I_{2, k}\\  } }
\abs{
m_n\pr{s,t_2}-m_n\pr{s,t_1}
   }.
\end{align*}
Again,
we restrict our considerations to the first summand
 and we note that it suffices to show that
\begin{align*}
\lim_{\delta\to 0}\limsup_{n\to \infty}\frac{1}{\delta}
\PP\pr{\sup_{ 1\leq s \leq R  } \sup_{\substack{t_2, t_1\in I_{1, k}}}\abs{
m_n\pr{s,t_2}-m_n\pr{s,t_1}
   }>\epsilon     }=0.
\end{align*}
Since, due to stationarity of the data-generating process,
\begin{align*}
\sup_{ 1\leq s \leq R  } \sup_{\substack{ t_2, t_1\in I_{1, k}}}\abs{
m_n\pr{s,t_2}-m_n\pr{s,t_1}
   }\overset{\mathcal{D}}{=}\sup_{ 1\leq s \leq R  } \sup_{\substack{t_2, t_1\in I_{1, 0}}}\abs{
m_n\pr{s,t_2}-m_n\pr{s,t_1}
   },
\end{align*}
verification of \eqref{eq:tightness_Mn_part2} follows by the same argument as  verification \eqref{eq:tightness_Mn_part1}.
\end{proof}

\subsection{Proof of Corollaries~\ref{cor:convergence_gamma_hat}, \ref{cor:convergence_gamma_Hill}, \ref{cor:convergence_test_statistic}, and \ref{cor:convergence_test_statistic_Hill}}

\begin{proof}[Proof of Corollary~\ref{cor:convergence_gamma_hat}]
An argument  from \cite{kulik:soulier:2011} is repeated and hence many technicalities are omitted. Note that the arguments below are model-free and only use  the conclusion of Theorem \ref{thm:TEP}. 

It holds that 
\begin{align*}
\widehat{\gamma}_{\lfloor nt\rfloor}=\frac{A_n(t)}{B_n(t)},
\end{align*}
where
\begin{align*}
&A_n(t)\defeq\frac{1}{n\overline{F}(u_n)}\sum_{j=1}^{\pe{nt}}\log\left(\frac{X_j}{u_n}\right)\1\ens{X_j>u_n} \ \text{and} \ B_n(t)\defeq\frac{1}{n\overline{F}(u_n)}\sum_{j=1}^{\lfloor nt\rfloor}\1\{X_j>u_n\} .
\end{align*}
By substracting and adding $t\alpha^{-1}/B_n\pr{t}$, 
the following equality holds:
\begin{align*} 
ta_n\left(\widehat{\gamma}_{\lfloor nt\rfloor}-\alpha^{-1}\right)
= \frac{a_nt}{B_n\pr{t}}\pr{A_n\pr{t}-t\alpha^{-1}}
+\frac{a_nt}{\alpha}\pr{\frac{t}{B_n\pr{t}}-1},
\end{align*}
 where $a_n=n/d_{n, r}$ 
if  $n/d_{n, r}=o\left(\sqrt{n\overline{F}(u_n)}\right)$
and  $a_n=\sqrt{n\overline{F}(u_n)}$ 
if $\sqrt{n\overline{F}(u_n)}=o\left(n/d_{n, r}\right)$.

We note that,  compared to \cite{kulik:soulier:2011}, the last term appears additionally due to the fact that we consider the two-parameter processes.

Rewriting $A_n\pr{t}-t\alpha^{-1}$ as an integral and 
replacing $\widetilde{T_n}-T$ by $e_n$, 
we have
 \begin{align}\label{eq:expression_for_gamma_n_integral}
t\left(\widehat{\gamma}_{\lfloor nt\rfloor}-\alpha^{-1}\right)
= \frac{a_nt}{B_n\pr{t}}
\int_{1}^\infty s^{-1}e_n\pr{s,t} ds
-\frac{a_nt}{\alpha B_n\pr{t}} e_n\pr{1,t}.
\end{align}

We 
will show  weak convergence 
of the sequence
\begin{equation}
 Y_n^{\pr{R}}\pr{t}:= \frac{a_nt}{B_n\pr{t}}
\int_{1}^R s^{-1}e_n \pr{s,t} ds
-\frac{a_nt}{\alpha B_n\pr{t}} e_n\pr{1,t}, \ n\geq 1, 
\end{equation}
 in $D[t_0,1]$ by an application of the continuous mapping theorem. For this,  
 we have to initially show that  terms of the form 
$t/B_n\pr{t}$ are negligible. Noting that $B_n(t)=\widetilde{T}_{n}(1,t)$, it follows 
by Theorem~\ref{thm:TEP} that 
\begin{equation}\label{eq:convergence_in_prob_Bn_sur_t}
\sup_{t_0\leq t\leq 1}
\abs{t\frac{B_n(1)}{B_n(t)}-1}\overset{P}{\longrightarrow}0.
\end{equation}
Indeed, Theorem~\ref{thm:TEP} implies that $\sup_{t_0\leq t\leq 1}
\abs{B_n\pr{t}-t}\to 0$ in probability and 
\begin{multline*}
 \sup_{t_0\leq t\leq 1}
\abs{t\frac{B_n(1)}{B_n(t)}-1}
\leq \frac 1{B_n\pr{t_0}}
 \sup_{t_0\leq t\leq 1}
\abs{tB_n(1)-B_n\pr{t}}\\
\leq \frac 1{B_n\pr{t_0}}
 \sup_{t_0\leq t\leq 1}
\abs{t\pr{B_n(1)-1}+t-B_n\pr{t}}\leq \frac 2{B_n\pr{t_0}}
\sup_{t_0\leq t\leq 1}
\abs{B_n\pr{t}-t}\to 0.
\end{multline*}
As a consequence, rewriting $Y_n^{\pr{R}}$ as
\begin{multline}
  Y_n^{\pr{R}}\pr{t}  
  = a_n
\int_{1}^R s^{-1}e_n \pr{s,t} ds
-\alpha^{-1}a_n e_n\pr{1,t}\\+  \pr{\frac{t}{B_n\pr{t}}-1}
\int_{1}^Ra_n s^{-1}e_n \pr{s,t} ds
-\pr{\frac{t}{ B_n\pr{t}}-1 } \alpha^{-1}a_ne_n\pr{1,t}
\end{multline}
and combining Theorem~\ref{thm:TEP} with \eqref{eq:convergence_in_prob_Bn_sur_t} shows that the limit of the sequence $Y_n^{\pr{R}}, n\geq 1$ corresponds to the limit of
\begin{equation}
 Z_n^{\pr{R}}\defeq a_n
\int_{1}^R s^{-1}e_n \pr{s,t} ds
-\alpha^{-1}a_n e_n\pr{1,t}, \ n\geq 1.
\end{equation}
By Theorem~\ref{thm:TEP} and the continuous mapping theorem, we conclude that $Z_n^{\pr{R}}$, $n\geq 1$, 
converges in distribution to $\int_{1}^R\xi \pr{s,t} ds-\alpha^{-1}\xi\pr{1,t}$,
where $\xi(s,t)$ it the limiting process in  \eqref{eq:limit-1} or \eqref{eq:limit-2}, respectively.
Using the same arguments as in \cite{kulik:soulier:2011}, it can be shown  that the convergence of $Y_n^{\pr{R}}$, $n\geq 1$,  can be easily extended to convergence 
of
\begin{equation}
\frac{a_nt}{B_n\pr{t}}
\int_{1}^{\infty} s^{-1}e_n \pr{s,t} ds
-\frac{a_nt}{\alpha B_n\pr{t}} e_n\pr{1,t}, \ n\geq 1.
\end{equation}
Hence, we conclude that 
\begin{equation*}
 ta_n\left(\widehat{\gamma}_{\lfloor nt\rfloor}-\alpha^{-1}\right)
 \Rightarrow \int_{1}^\infty \xi \pr{s,t} ds-\alpha^{-1}\xi\pr{1,t}
\end{equation*}
in $D[t_0, 1]$.

If $\frac{n}{d_{n, r}}=o\left(\sqrt{n\overline{F}(u_n)}\right)$, separation of the variables 
$s$ and $t$ shows that the limit vanishes. This  finishes the 
proof of Corollary~\ref{cor:convergence_gamma_hat}.
\end{proof}

\begin{proof}[Proof of Corollary~\ref{cor:convergence_gamma_Hill}]
Define
\begin{align*}
\widehat{T}_n(s, t)\defeq \frac{1}{\lfloor k_nt\rfloor}\sum\limits_{j=1}^{\lfloor nt\rfloor}\1\ens{X_j>sX_{\lfloor nt\rfloor:\lfloor nt\rfloor-\lfloor k_nt\rfloor}}.
\end{align*}
Then, it holds that
\begin{align*}
\int_1^{\infty}s^{-1}\widehat{T}_n(s, t)ds =\frac{1}{\lfloor k_nt\rfloor}\sum\limits_{j=1}^{\lfloor nt\rfloor}\int_1^{\infty}s^{-1}\1\ens{X_j>sX_{\lfloor nt\rfloor:\lfloor nt\rfloor-\lfloor k_nt\rfloor}}ds
=\widehat{\gamma}_{\rm Hill}(t).
\end{align*}
According to Skorokhod's representation theorem (Theorem~2.3.4 in \cite{shorack:wellner:1986}) and Theorem \ref{thm:TEP}, we may assume without loss of generality that
\begin{align*}
\sup\limits_{s\in [1, \infty], t\in [0, 1]}\left|a_n\left(\frac{1}{n\overline{F}(u_n)}\sum\limits_{j=1}^{\lfloor nt\rfloor}\1\ens{X_j> u_ns}-s^{-\alpha}t\right)-\xi(s, t)\right|\longrightarrow 0 \ \ \text{almost surely,}
\end{align*}
where $a_n=n/d_{n, r}$ 
if  $n/d_{n, r}=o\left(\sqrt{n\overline{F}(u_n)}\right)$,
 $a_n=\sqrt{n\overline{F}(u_n)}$ 
if $\sqrt{n\overline{F}(u_n)}=o\left(n/d_{n, r}\right)$,
 and $\xi$ denotes the corresponding limiting process in Theorem \ref{thm:TEP}.
In order to apply Vervaat's Lemma as stated by Lemma 5 in \cite{einmahl:2010}, we rephrase the above convergence as:
\begin{align}\label{eq:Skorohod}
\sup\limits_{s\in [0, 1], t\in [t_0, 1]}\left|a_n\left(\Gamma_{n, t}(s)-s\right)-\frac{1}{t}\xi(s^{-\frac{1}{\alpha}}, t)\right|\longrightarrow 0 \ \ \text{almost surely}
\end{align}
for any $t_0>0$ and with
\begin{align*}
\Gamma_{n, t}(s)\defeq \frac{1}{\overline{F}(u_n)}\frac{1}{\lfloor nt\rfloor}\sum\limits_{j=1}^{\lfloor nt\rfloor}\1\ens{
 X_j>u_ns^{-\frac{1}{\alpha}} }.
\end{align*}
Choosing $k_n=n\overline{F}(u_n)$, it follows that
\begin{align*}
\Gamma_{n, t}^{-}(s)=\left(u_n^{-1}X_{\lfloor nt\rfloor:\lfloor nt\rfloor-s\lfloor k_nt\rfloor }\right)^{-\alpha}.
\end{align*}
As a result, Vervaat's Lemma yields
\begin{align}\label{eq:vervaat}
&\sup\limits_{s\in [0, 1], t\in [t_0, 1]}\left|a_n\left(\left(u_n^{-1}X_{\lfloor nt\rfloor:\lfloor nt\rfloor-s\lfloor k_nt\rfloor}\right)^{-\alpha}-s\right)+\frac{1}{t}\xi(s^{-\frac{1}{\alpha}}, t)\right|
\longrightarrow 0 \  \ \text{almost surely}.
\end{align}
Setting $s_n=su_n^{-1}X_{\lfloor nt\rfloor:\lfloor nt\rfloor-\lfloor k_nt\rfloor}$, we arrive at
\begin{align*}
&\sup\limits_{s\in [1, \infty], t\in [t_0, 1]}\left|a_n\left(\frac{1}{n\overline{F}(u_n)}\frac{1}{t}\sum\limits_{j=1}^{\lfloor nt\rfloor}\1\ens{ X_j> sX_{\lfloor nt\rfloor:\lfloor nt\rfloor-\lfloor k_nt\rfloor} }-s^{-\alpha}\right)-\frac{1}{t}\left(\xi(s, t)-s^{-\alpha}\xi(1, t)\right)\right|\\
\leq
&\sup\limits_{s\in [1, \infty], t\in [t_0, 1]}\left|a_n\left(\frac{1}{n\overline{F}(u_n)}\frac{1}{t}\sum\limits_{j=1}^{\lfloor nt\rfloor}\1\ens{ X_j> u_ns }-s^{-\alpha}\right)-\frac{1}{t}\xi(s, t)\right|\\
&+\sup\limits_{s\in [1, \infty], t\in [t_0, 1]}\left|a_n\left(s_n^{-\alpha}-s^{-\alpha}\right)+s^{-\alpha}\frac{1}{t}\xi(1, t)\right|+\sup\limits_{s\in [1, \infty], t\in [t_0, 1]}\left|\frac{1}{t}\left(\xi(s_n, t)-\xi(s, t)\right)\right|.
\end{align*}
The first two summands on the right-hand side converge to $0$ almost surely according to
 \eqref{eq:Skorohod} and \eqref{eq:vervaat}.
Moreover, we have $s_n^{-\alpha}=s^{-\alpha}\left(1+o(1)\right)$ a.s. uniformly in $t$ and $s$ such that it follows by a continuity argument that the third summand converges to $0$ almost surely, as well.

Since
$k_n=n\overline{F}(u_n)$, we have
\begin{align*}
a_n\left(\widehat{T}_n(s, t)-T(s, 1)\right)
&=a_n\left(\frac{1}{\lfloor k_nt\rfloor}\sum\limits_{j=1}^{\lfloor nt\rfloor}\1\ens{ X_j>sX_{\lfloor nt\rfloor:\lfloor nt\rfloor-\lfloor k_nt\rfloor} }-s^{-\alpha}\right)\\
&=a_n\left(\frac{1}{n\overline{F}(u_n)}\frac{1}{t}\sum\limits_{j=1}^{\lfloor nt\rfloor}\1\ens{ X_j> sX_{\lfloor nt\rfloor:\lfloor nt\rfloor-\lfloor k_nt\rfloor} }-s^{-\alpha}\right)+o_P(1)\\
&\Rightarrow \frac{1}{t}\left(\xi(s, t)-s^{-\alpha}\xi(1, t)\right).
\end{align*}
Similar to the proof of Corollary \ref{cor:convergence_gamma_hat},  it follows that
\begin{align*}
a_nt\left(\widehat{\gamma}_{\rm Hill}(t)-\gamma\right)
=t\int_1^{\infty}s^{-1}a_n\left(\widehat{T}_n(s,t)-T(s, 1)\right)ds
\Rightarrow \int_{1}^\infty \xi \pr{s,t} ds-\alpha^{-1}\xi\pr{1,t}
\end{align*}
in $D[t_0, 1]$.
\end{proof}

\begin{proof}[Proof of Corollaries~\ref{cor:convergence_test_statistic} and \ref{cor:convergence_test_statistic_Hill}]
Since, due to Corollaries~\ref{cor:convergence_gamma_hat} and \ref{cor:convergence_gamma_Hill}, $\widehat{\gamma}_{n}$ and $\widehat{\gamma}_{\text{Hill}}(1)$ converge to $\gamma$ in probability,
 Corollaries~\ref{cor:convergence_test_statistic} and \ref{cor:convergence_test_statistic_Hill} follow from Slutsky's theorem and an application of the continuous mapping theorem to
the processes
\begin{align*}
a_nt\left(\widehat{\gamma}_{\pe{nt}}-\gamma\right)&\Rightarrow  \int_1^{\infty}s^{-1}\xi(s,t)ds-\alpha^{-1}\xi(1, t)
, \ t\in [t_0, 1],\\
a_nt\left(\widehat{\gamma}_{\text{Hill}}(t)-\gamma\right)&\Rightarrow  \int_1^{\infty}s^{-1}\xi(s,t)ds-\alpha^{-1}\xi(1, t)
, \ t\in [t_0, 1],
\end{align*}
and the function
\begin{align*}
f\mapsto \sup\limits_{t\in [t_0, 1]}\left|f(t)-tf(1)\right|.
\end{align*}
\end{proof}

\section{Appendix: Second-order regular variation}
The following lemmata
originate in \cite{kulik:soulier:2011} and are essential to the proof of Theorem \ref{thm:TEP}.
\begin{lemma}[Lemma~4.1 in \cite{kulik:soulier:2011}] \label{lem:control_ratio_Fz}
Assume that
$\overline{F}_{\varepsilon} \in \text{2RV}\pr{\alpha, \eta^*}$.
For positive $\varepsilon$ there exists a constant $C$ such that
\begin{equation}\label{eq:control_ratio_Fz}
\forall t\geq 1, \forall z>0, \quad
\abs{\frac{\overline{F}_{\varepsilon}\pr{zt}}{\overline{F}_{\varepsilon}\pr{t}} -z^{-\alpha}  }\leq
C\eta^*\pr{t}z^{-\alpha-\rho}\pr{\max\ens{z,z^{-1}}}^\varepsilon.
\end{equation}
\end{lemma}
This implies that
\begin{equation}\label{eq:2ndrv}
\sup_{s\geq s_0}\left|\frac{\overline{F}_{\varepsilon}(u_ns)}{\overline{F}_{\varepsilon}(u_n)}-s^{-\alpha}\right| =O(\eta^*{(u_n)}).
\end{equation}

Using boundedness of $\eta^*$, we get the following simplified version of inequality
\eqref{eq:control_ratio_Fz}:
\begin{equation}\label{eq:control_ratio_Fz_simplified}
\forall t\geq 1, \forall z>0, \quad
 \frac{\overline{F}_{\varepsilon}\pr{zt}}{\overline{F}_{\varepsilon}\pr{t}}  \leq z^{-\alpha} +
C_\varepsilon z^{-\alpha-\rho}\pr{\max\ens{z,z^{-1}}}^\varepsilon.
\end{equation}

We also need the following bound on the increments of $\overline{F}_{\varepsilon}$.

\begin{lemma}[Lemma~4.2 in \cite{kulik:soulier:2011}] \label{lem:control_ratio_difference_Fz}
Assume that
$\overline{F}_{\varepsilon} \in \text{2RV}\pr{\alpha, \eta^*}$.
For positive $\epsilon$ there exists a constant $C$ such that
for all $t\geq 1$ and $b>a>0$
\begin{equation}\label{eq:control_ratio_difference_Fz}
\abs{
\frac{\overline{F}_{\varepsilon}\pr{at}-\overline{F}_{\varepsilon}\pr{bt}   }{\overline{F}_{\varepsilon}\pr{t}}
-\pr{a^{-\alpha}-b^{-\alpha}}
}\leq
C\eta^*\pr{t}\pr{\min\ens{a, 1}}^{-\alpha-\rho-\epsilon}\pr{b-a}.
\end{equation}
\end{lemma}
Using again boundedness of $\eta^*$, we get the following simplified version of inequality
\eqref{eq:control_ratio_difference_Fz}:
\begin{equation}\label{eq:control_ratio_difference_Fz_simplified}
\frac{\overline{F}_{\varepsilon}\pr{at}-\overline{F}_{\varepsilon}\pr{bt}   }{\overline{F}_{\varepsilon}\pr{t}}
\leq a^{-\alpha}-b^{-\alpha} +
C_\epsilon\pr{\min\ens{a,1}}^{-\alpha-\rho-\epsilon}\pr{b-a}.
\end{equation}

\bibliographystyle{apalike}
\bibliography{change_point_tests_for_the_tail_parameter_of_LMSV_time_series}

\end{document}